\documentclass[10pt]{article}%
\usepackage{amsmath}
\usepackage{amsfonts}
\usepackage{amssymb}
\usepackage{amsthm}
\usepackage{graphicx}%
\providecommand{\U}[1]{\protect\rule{.1in}{.1in}}
\newtheorem{theorem}{Theorem}[section]

\newtheorem{corollary}[theorem]{Corollary}
\newtheorem{lemma}[theorem]{Lemma}
\newtheorem{proposition}[theorem]{Proposition}
\theoremstyle{definition}
\newtheorem{definition}[theorem]{Definition}
\theoremstyle{remark}
\newtheorem{example}[theorem]{Example}
\newtheorem{remark}[theorem]{Remark}
\renewenvironment{proof}[1][Proof]{\noindent\textbf{#1.} }{\ \rule{0.5em}{0.5em}}
\numberwithin{equation}{section}
\begin{document}

\title{Modified spectral parameter power series representations for solutions of
Sturm-Liouville equations and their applications}
\author{Vladislav V. Kravchenko and Sergii M. Torba\\{\small Departamento de Matem\'{a}ticas, CINVESTAV del IPN, Unidad
Quer\'{e}taro, }\\{\small Libramiento Norponiente No. 2000, Fracc. Real de Juriquilla, }\\{\small Quer\'{e}taro, Qro. C.P. 76230 MEXICO}\\{\small
e-mail: vkravchenko@math.cinvestav.edu.mx}\\
{\small e-mail: storba@math.cinvestav.edu.mx
\thanks{Research was supported by CONACYT,
Mexico via the project 166141.}}}
\maketitle

\begin{abstract}
Spectral parameter power series (SPPS) representations for solutions of
Sturm-Liouville equations proved to be an efficient practical tool for solving
corresponding spectral and scattering problems. They are based on a
computation of recursive integrals, sometimes called formal powers. In this
paper new relations between the formal powers are presented which considerably
improve and extend the application of the SPPS method. For example, originally
the SPPS method at a first step required to construct a nonvanishing (in
general, a complex-valued) particular solution corresponding to the zero-value
of the spectral parameter. The obtained relations remove this limitation.
Additionally, equations with \textquotedblleft nasty\textquotedblright
 Sturm-Liouville coefficients $1/p$ or $r$ can be solved by the SPPS method.

We develop the SPPS representations for solutions of Sturm-Liouville equations
of the form
\[
\left(  p(x)u^{\prime}\right)  ^{\prime}+q(x)u=%
{\displaystyle\sum\limits_{k=1}^{N}}
\lambda^{k}R_{k}\left[  u\right]  ,\quad x\in(a,b)
\]
where $R_{k}\left[  u\right]  :=r_{k}(x)u+s_{k}(x)u^{\prime}$, $k=1,\ldots N$,
the complex-valued functions $p$, $q$, $r_{k}$, $s_{k}$ are continuous on the
finite segment $\left[  a,b\right]  $.

Several numerical examples illustrate the efficiency of the method and its
wide applicability.

\end{abstract}

\section{Introduction}

Solutions of sufficiently regular linear second order Sturm-Liouville
equations considered as functions of a spectral parameter are entire functions
which in particular means that they admit a normally convergent Taylor series
representation in terms of the spectral parameter in the whole complex plane.
The coefficients of the series are functions of the independent variable. For
example, in the simplest case of the equation $y^{\prime\prime}(x)=\lambda
y(x)$ two linearly independent solutions (satisfying in the origin the initial
conditions $(1,0)$, $(0,1)$) can be chosen in the form $y_{1}(x)=\cosh
\sqrt{\lambda}x$ and $y_{2}(x)=\left(  \sinh\sqrt{\lambda}x\right)
/\sqrt{\lambda}$. The Taylor coefficients in their power series in terms of
the spectral parameter $\lambda$ with the center $\lambda=0$ are powers of the
independent variable divided by corresponding factorials $x^{2n}/(2n)!$ and
$x^{2n+1}/(2n+1)!$ respectively.

In \cite{KrCV08} a simple way for calculating the Taylor coefficients for
spectral parameter power series (SPPS) defining solutions of the
Sturm-Liouville equation $(pu^{\prime})^{\prime}+qu=\lambda u$ was proposed,
based on the theory of complex pseudoanalytic functions. In
\cite{KrPorter2010} (see also \cite{APFT}) that result was extended onto
equations of the form
\begin{equation}
(pu^{\prime})^{\prime}+qu=\lambda ru\label{IntroSL}%
\end{equation}
and proved in a simpler way with no need of pseudoanalytic function theory
(see Theorem \ref{ThSolSL} below). The Taylor coefficients in the SPPS
representations are calculated as recursive integrals and called formal
powers. The SPPS representations found numerous applications, see two recent review papers \cite{KKRosu}, \cite{KT Obzor}. In \cite{KKB} SPPS representations were obtained for solutions of
fourth order Sturm-Liouville equations of the form
\begin{equation*}
(  pu^{\prime\prime})^{\prime\prime}+(  qu^{\prime})
^{\prime}=\lambda R\left[  u\right]  \label{SL4OrderR}%
\end{equation*}
where $R$ is a linear differential operator of the order $n\leq3$, and in
\cite{CKT2013} for Bessel-type singular Sturm-Liouville equations. In \cite{KTV} the SPPS representations were obtained for equations of the form
\begin{equation*}
(  p(x)u^{\prime})^{\prime}+q(x)u=\sum_{k=1}^{N}
\lambda^{k}r_{k}(x)u
\end{equation*}
and used for studying spectral problems for Zakharov-Shabat systems.

In \cite{CKT} it was shown that at least in the case of the one-dimensional
Schr\"{o}dinger equation
\begin{equation}
u^{\prime}{}^{\prime}+qu=\lambda u\label{IntroSchr}%
\end{equation}
the formal powers are the images of usual powers $x^{k}$, $k=0,1,2,\ldots$
under the action of a corresponding transmutation operator. In \cite{KrT2013}
based on this observation a new method for solving spectral problems for
(\ref{IntroSchr}) was developed. The method possesses a remarkable unique
feature: it allows one to compute thousands of eigendata with a non-decreasing
accuracy. In \cite{CKR}, \cite{CCK}, \cite{CKT} and \cite{KKTT} methods for
solving different problems for partial differential equations involving the
computation of formal powers were developed.

Thus, the computation of formal powers is required for application of
different methods and in different models. An important restriction for
computing formal powers as proposed in \cite{KrCV08}, \cite{KrPorter2010} and
further publications consisted in the necessity of a nonvanishing particular
solution of the equation
\begin{equation}
(pv^{\prime})^{\prime}+qv=0.\label{IntroSL0}%
\end{equation}
When $p$ and $q$ are real valued (and sufficiently regular) such nonvanishing
solution can be proposed in the form $v_{0}=v_{1}+iv_{2}$ where $v_{1}$ and
$v_{2}$ are arbitrary linearly independent real-valued solutions of
(\ref{IntroSL0}). However for complex-valued coefficients $p$ and $q$ there is
no such simple way for its construction. Moreover, even when $v_{0}$ does not
vanish but in some points is relatively close to zero, the computation of
formal powers may present difficulties.

In the present work we solve two problems. 1) We develop an SPPS
representation which is not limited to nonvanishing particular solutions of
auxiliary equations and admits certain \textquotedblleft
nastiness\textquotedblright\ in the coefficients. For example, $p$ is allowed
to have zeros. 2) We extend the SPPS method onto equations of the form
\begin{equation}
(  p(x)u^{\prime})^{\prime}+q(x)u=\sum_{k=1}^{N}
\lambda^{k}R_{k}\left[  u\right]  ,\quad x\in(a,b)\label{IntroPencil}%
\end{equation}
where $R_{k}\left[  u\right]  :=r_{k}(x)u+s_{k}(x)u^{\prime}$, $k=1,\ldots N$,
the complex-valued functions $p$, $q$, $r_{k}$, $s_{k}$ are continuous on the
finite segment $\left[  a,b\right]  $. The presented numerical results show
that nowadays this is one of the most accurate ways for solving corresponding spectral problems with a wide range of applicability (e.g., few available algorithms are applicable to complex coefficients, complex spectra, polynomial pencils of operators, etc.).

In Section \ref{Section2} we prove new relations concerning formal powers and obtain the
modified SPPS representations for Sturm-Liouville equations of the form
(\ref{IntroSL}). In Section \ref{Section3} we extend this result onto equations of the form
(\ref{IntroPencil}). In Section \ref{Section4} we describe the algorithm and the numerical
implementation of the proposed method for solving spectral problems and give
eight numerical examples illustrating its performance.

\section{SPPS representations}\label{Section2}
\subsection{The original SPPS representation}

In \cite{KrPorter2010} the following theorem was proved.

\begin{theorem}[SPPS representation, \cite{KrPorter2010}] \label{ThSolSL} Assume that on a
finite segment $[a,b]$, equation
\begin{equation}
(pv^{\prime})^{\prime}+qv=0, \label{SL0}%
\end{equation}
possesses a particular solution $f$ such that the functions $f^{2}r$ and
$1/(f^{2}p)$ are continuous on $[a,b]$. Then the general solution of the
equation
\begin{equation}
(pu^{\prime})^{\prime}+qu=\lambda ru \label{SL}%
\end{equation}
on $(a,b)$ has the form
\begin{equation}
u=c_{1}u_{1}+c_{2}u_{2} \label{genmain}%
\end{equation}
where $c_{1}$ and $c_{2}$ are arbitrary complex constants,
\begin{equation}
u_{1}=f\sum_{k=0}^{\infty}
\lambda^{k}\widetilde{X}^{(2k)}\quad\text{and}\quad u_{2}=f\sum_{k=0}^{\infty}
\lambda^{k}X^{(2k+1)} \label{gensol}%
\end{equation}
with $\widetilde{X}^{(n)}$ and $X^{(n)}$ being defined by the recursive
relations $\widetilde{X}^{(-n)}\equiv X^{(-n)}\equiv0$ for $n\in\mathbb{N}$,
\begin{align}
\widetilde{X}^{(0)}& \equiv1,\qquad X^{(0)}\equiv1, \label{Xgen1}\\
\widetilde{X}^{(n)}(x)&=
\begin{cases}
\displaystyle\int_{x_{0}}^{x}
\widetilde{X}^{(n-1)}(s)f^{2}(s)r(s)\,ds, & n \text{ odd,}\smallskip \\
\displaystyle\int_{x_{0}}^{x}
\widetilde{X}^{(n-1)}(s)\dfrac{1}{f^{2}(s)p(s)}\,ds, & n \text{ even,}
\end{cases} \label{Xgen2}\\
X^{(n)}(x)&=
\begin{cases}
\displaystyle\int_{x_{0}}^{x}
X^{(n-1)}(s)\dfrac{1}{f^{2}(s)p(s)}\,ds, & n \text{ odd,}\smallskip \\
\displaystyle\int_{x_{0}}^{x}
X^{(n-1)}(s)f^{2}(s)r(s)\,ds, & n \text{ even},
\end{cases}\label{Xgen3}
\end{align}
where $x_{0}$ is an arbitrary point in $[a,b]$ such that $p$ is continuous at
$x_{0}$ and $p(x_{0})\neq0$. Further, both series in \eqref{gensol} converge
uniformly on $[a,b]$.

The solutions $u_{1}$ and $u_{2}$ satisfy the initial conditions
\begin{align*}
u_{1}(x_{0})&=f(x_{0}),& u_{1}^{\prime}(x_{0})&=f^{\prime}(x_{0}),\\
u_{2}(x_{0})&=0,& u_{2}^{\prime}(x_{0})&=\frac{1}{f(x_{0})p(x_{0})}.
\end{align*}
\end{theorem}

This result was first obtained in \cite{KrCV08} with the aid of pseudoanalytic
function theory \cite{APFT} and for the case $r\equiv1$. The functions
$\widetilde{X}^{(n)}$ and $X^{(n)}$ are called \emph{formal powers} since they
generalize the usual powers $(x-x_{0})^{n}$ or more precisely $(x-x_{0}%
)^{n}/n!$ (when $f\equiv p\equiv r\equiv1$).

\subsection{Relations between formal powers associated with two different particular solutions}

Now let us suppose additionally that $f(x_{0})=1$ and that together with $f$
there exists another linearly independent solution $g$ of (\ref{SL0})
satisfying the same conditions as $f$ and such that $g(x_{0})=1$. Then one can
construct formal powers corresponding to $g$. Let us denote them by
$\widetilde{Y}^{(n)}$ and $Y^{(n)}$ correspondingly. Thus, $\widetilde
{Y}^{(-n)}\equiv Y^{(-n)}\equiv0$ for $n\in\mathbb{N}$,
\begin{align*}
\widetilde{Y}^{(0)}&\equiv1,\qquad Y^{(0)}\equiv1,\\ 
\widetilde{Y}^{(n)}(x)&=
\begin{cases}
\displaystyle\int_{x_{0}}^{x}
\widetilde{Y}^{(n-1)}(s)g^{2}(s)r(s)\,ds, & n \text{ odd,}\smallskip\\
\displaystyle\int_{x_{0}}^{x}
\widetilde{Y}^{(n-1)}(s)\dfrac{1}{g^{2}(s)p(s)}\,ds, & n \text{ even,}
\end{cases}\\ 
Y^{(n)}(x)&=
\begin{cases}
\displaystyle\int_{x_{0}}^{x}
Y^{(n-1)}(s)\dfrac{1}{g^{2}(s)p(s)}\,ds, & n \text{ odd,}\smallskip \\
\displaystyle\int_{x_{0}}^{x}
Y^{(n-1)}(s)g^{2}(s)r(s)\,ds, & n \text{ even}.
\end{cases}
\end{align*}

Later on we will show that the restrictions imposed on $f$ and $g$ can be
relaxed. At this moment we need them to establish relations between the two
sets of formal powers. Denote $\rho=\frac{1}{p(x_{0})(g^{\prime}%
(x_{0})-f^{\prime}(x_{0}))}$.

\begin{proposition}
\label{Prop Relations Formal Powers} Assume that on a finite interval
$[a,b]$, equation \eqref{SL0} possesses two particular solutions $f$ and
$g$ such that $f(x_{0})=g(x_{0})=1$, $x_{0}$ is an arbitrary point in $[a,b]$
such that $p$ is continuous at $x_{0}$ and $p(x_{0})\neq0$, the functions
$f^{2}r$, $1/(f^{2}p)$, $g^{2}r$ and $1/(g^{2}p)$ are continuous on $[a,b]$.
Then the following relations hold
\begin{align}
gY^{(2k+1)}  &  =fX^{(2k+1)}\label{R1A}\\
&  =\rho\bigl(  g\widetilde{Y}^{(2k)}-f\widetilde{X}^{(2k)}\bigr)
\label{R1B}\\
&  =\rho\bigl(  gX^{(2k)}-fY^{(2k)}\bigr)  , \label{R2}\\
g\widetilde{Y}^{(2k)}& =gX^{(2k)}+\rho\bigl(  g\widetilde{X}^{(2k-1)}%
-f\widetilde{Y}^{(2k-1)}\bigr), \label{R3}\\
f\widetilde{X}^{(2k)}&=fY^{(2k)}+\rho\bigl(  g\widetilde{X}^{(2k-1)}
-f\widetilde{Y}^{(2k-1)}\bigr)  \label{R4}%
\end{align}
for any $k=0,1,2,\ldots$.
\end{proposition}

\begin{proof}
Consider two pairs of linearly independent solutions of (\ref{SL}) constructed
according to Theorem \ref{ThSolSL}. One pair is generated by the particular
solution $f$ and has the form (\ref{gensol}) meanwhile the second pair is
generated by $g$ and has the form
\[
v_{1}=g%
{\displaystyle\sum\limits_{k=0}^{\infty}}
\lambda^{k}\widetilde{Y}^{(2k)}\quad\text{and}\quad v_{2}=g%
{\displaystyle\sum\limits_{k=0}^{\infty}}
\lambda^{k}Y^{(2k+1)}.
\]
Due to Theorem \ref{ThSolSL} the solutions $v_{1}$ and $v_{2}$ satisfy the
initial conditions $v_{1}(x_{0})=g(x_{0})$, $v_{1}^{\prime}(x_{0})=g^{\prime
}(x_{0})$, $v_{2}(x_{0})=0$, $v_{2}^{\prime}(x_{0})=\frac{1}{g(x_{0})p(x_{0}%
)}$. Since $f(x_{0})=g(x_{0})=1$, we obtain $u_{2}\equiv v_{2}$. From the
equality of the corresponding series \eqref{gensol} for any value of the parameter $\lambda$
we obtain (\ref{R1A}).

Comparison of the initial conditions gives us also the following relation
\[
v_{1}=u_{1}+\frac{1}{\rho}u_{2}.
\]
Thus,
\[
g%
{\displaystyle\sum\limits_{k=0}^{\infty}}
\lambda^{k}\widetilde{Y}^{(2k)}=f%
{\displaystyle\sum\limits_{k=0}^{\infty}}
\lambda^{k}\widetilde{X}^{(2k)}+\frac{1}{\rho}f%
{\displaystyle\sum\limits_{k=0}^{\infty}}
\lambda^{k}X^{(2k+1)}%
\]
for any $\lambda\in\mathbb{C}$. Hence for any $k=0,1,2,\ldots$ we have
\[
g\widetilde{Y}^{(2k)}=f\left(  \widetilde{X}^{(2k)}+\frac{1}{\rho}%
X^{(2k+1)}\right)
\]
from where (\ref{R1B}) follows.

Consider the equality $u_{2}^{\prime}\equiv v_{2}^{\prime}$. It implies the
equality of the series
\[
f^{\prime}\sum_{k=0}^{\infty}\lambda^{k}X^{(2k+1)}+\frac{1}{fp}
\sum_{k=0}^{\infty}
\lambda^{k}X^{(2k)}=g^{\prime}\sum_{k=0}^{\infty}
\lambda^{k}Y^{(2k+1)}+\frac{1}{gp}\sum_{k=0}^{\infty}
\lambda^{k}Y^{(2k)}
\]
and hence
\[
f^{\prime}X^{(2k+1)}+\frac{1}{fp}X^{(2k)}=g^{\prime}Y^{(2k+1)}+\frac{1}%
{gp}Y^{(2k)}%
\]
for any $k=0,1,2,\ldots$. From (\ref{R1A}) we have $g^{\prime}Y^{(2k+1)}%
=\frac{g^{\prime}}{g}fX^{(2k+1)}$ and consequently,
\[
\left(  f^{\prime}-\frac{g^{\prime}}{g}f\right)  X^{(2k+1)}=\frac{1}{p}\left(
\frac{1}{g}Y^{(2k)}-\frac{1}{f}X^{(2k)}\right)  .
\]
Notice that by Liouville's formula for the Wronskian
\begin{equation}
g^{\prime}f-gf^{\prime}=W(f,g)=\frac{p(x_0)}{p}W(f,g)(x_0)=\frac{1}{\rho p}. \label{W[f,g]}%
\end{equation}
Then
\[
\frac{1}{\rho}X^{(2k+1)}=\frac{g}{f}X^{(2k)}-Y^{(2k)}%
\]
from where we obtain (\ref{R2}).

Consider the equality $v_{1}^{\prime}=u_{1}^{\prime}+\frac{1}{\rho}%
u_{2}^{\prime}$. It can be written in the form
\begin{align*}
g^{\prime}%
\sum_{k=0}^{\infty}
\lambda^{k}\widetilde{Y}^{(2k)}+\frac{1}{gp}\sum_{k=1}^{\infty}
\lambda^{k}\widetilde{Y}^{(2k-1)}  &  =f^{\prime}
\sum_{k=0}^{\infty}
\lambda^{k}\widetilde{X}^{(2k)}+\frac{1}{fp}\sum_{k=1}^{\infty}
\lambda^{k}\widetilde{X}^{(2k-1)}\\
&\qquad  +\frac{1}{\rho}\left(  f^{\prime}\sum_{k=0}^{\infty}
\lambda^{k}X^{(2k+1)}+\frac{1}{fp}\sum_{k=0}^{\infty}
\lambda^{k}X^{(2k)}\right)
\end{align*}
which leads to the equality%
\[
g^{\prime}\widetilde{Y}^{(2k)}+\frac{1}{gp}\widetilde{Y}^{(2k-1)}=f^{\prime
}\widetilde{X}^{(2k)}+\frac{1}{fp}\widetilde{X}^{(2k-1)}+\frac{1}{\rho}\left(
f^{\prime}X^{(2k+1)}+\frac{1}{fp}X^{(2k)}\right)
\]
for any $k=0,1,2,\ldots$. Using (\ref{R1B}) we obtain%
\[
g^{\prime}\widetilde{Y}^{(2k)}+\frac{1}{gp}\widetilde{Y}^{(2k-1)}=\frac{1}%
{fp}\widetilde{X}^{(2k-1)}+\frac{f^{\prime}g}{f}\widetilde{Y}^{(2k)}+\frac
{1}{\rho fp}X^{(2k)}.
\]
Thus,
\[
\left(  g^{\prime}-\frac{f^{\prime}}{f}g\right)  \widetilde{Y}^{(2k)}=\frac
{1}{p}\left(  \frac{1}{f}\widetilde{X}^{(2k-1)}-\frac{1}{g}\widetilde
{Y}^{(2k-1)}+\frac{1}{\rho f}X^{(2k)}\right)  ,
\]
and taking into account (\ref{W[f,g]}) we arrive at (\ref{R3}). Finally,
(\ref{R4}) is the same (\ref{R3}) where $g$ plays the role of $f$ and vice versa.
\end{proof}

\subsection{Modified SPPS representation}
The relations between formal powers established in Proposition
\ref{Prop Relations Formal Powers} suggest another way for defining the formal
powers and formulating the SPPS representations for solutions of the
Sturm-Liouville equation.

\begin{definition}
\label{Def Systems Fn}Let equation \eqref{SL0} admit two linearly independent
solutions $f$ and $g$ such that $\left\{  f,\,g,\,pf^{\prime},\,pg^{\prime
}\right\}  \subset C^{1}[a,b]$ and $f(x_{0})=g(x_{0})=1$ where $x_{0}$ is any
point of $[a,b]$ such that $p(x_{0})\neq0$. Then the following systems of
functions $\{F_{n}\}$, $\{\widetilde{F}_{n}\}$,
$\{G_{n}\}$, $\{\widetilde{G}_{n}\}$ are defined
recursively as follows
\begin{align}
F_{-n}&\equiv G_{-n}\equiv\widetilde{F}_{-n}\equiv\widetilde{G}_{-n}
\equiv0\qquad\text{for }n\in\mathbb{N}, \label{D-}\\
F_{0}&\equiv G_{0}\equiv1, \qquad\widetilde{F}_{0}\equiv f, \qquad
\widetilde{G}_{0}\equiv g, \label{D0}
\end{align}
for an odd $n$:
\begin{align}
F_{n}&=G_{n}=\rho\left(  gF_{n-1}-fG_{n-1}\right)  , \label{D2}\\
\widetilde{F}_{n}(x)&=\int_{x_{0}}^{x}
\widetilde{F}_{n-1}(s)f(s)r(s)\,ds, \label{D3}\\
\widetilde{G}_{n}(x)&=\int_{x_{0}}^{x}
\widetilde{G}_{n-1}(s)g(s)r(s)\,ds, \label{D4}
\end{align}
and for an even $n$:
\begin{align}
F_{n}(x)&=\int_{x_{0}}^{x}
F_{n-1}(s)f(s)r(s)\,ds, \label{D5}\\
G_{n}(x)&=\int_{x_{0}}^{x}
G_{n-1}(s)g(s)r(s)\,ds, \label{D6}\\
\widetilde{F}_{n}&=fG_{n}-\rho\bigl(  f\widetilde{G}_{n-1}-g\widetilde{F}
_{n-1}\bigr), \label{D7}\\
\widetilde{G}_{n}&=gF_{n}-\rho\bigl(  f\widetilde{G}_{n-1}-g\widetilde{F}%
_{n-1}\bigr). \label{D8}
\end{align}
\end{definition}

Notice that from (\ref{D7}) and (\ref{D8}) we have that
\begin{equation*}
\widetilde{G}_{2n}-\widetilde{F}_{2n}=gF_{2n}-fG_{2n} \label{D9}%
\end{equation*}
and hence from (\ref{D2}) we obtain the relation
\begin{equation}
F_{2n+1}=G_{2n+1}=\rho\bigl(  \widetilde{G}_{2n}-\widetilde{F}_{2n}\bigr).\label{D10}
\end{equation}

\begin{remark}
\label{Rem Relations Old and New}It is easy to see that when additionally the
function $1/(f^{2}p)$ is continuous on $[a,b]$ and hence the systems of
functions $\{X^{(n)}\}$, $\{\widetilde{X}^{(n)}\}$ can be constructed, the following relations hold
\[
F_{n}=fX^{(n)}\quad\text{and}\quad\widetilde{F}_{n}=\widetilde{X}^{(n)}%
\qquad\text{for an odd }n
\]
and
\[
F_{n}=X^{(n)}\quad\text{and}\quad\widetilde{F}_{n}=f\widetilde{X}^{(n)}%
\qquad\text{for an even }n.
\]
\end{remark}

In the following lemma we prove several properties of the introduced functions.

\begin{lemma}
\label{Lemma Derivatives}For the functions defined by Definition
\ref{Def Systems Fn} the following relations hold.

For an odd $n$:
\begin{gather}
F_{n}^{\prime}=G_{n}^{\prime}=\rho\left(  g^{\prime}F_{n-1}-f^{\prime}%
G_{n-1}\right)  , \label{Fprime}\\
\left(  pF_{n}^{\prime}\right)  ^{\prime}+qF_{n}=rF_{n-2}, \label{Fbiprime}\\
\left(  pG_{n}^{\prime}\right)  ^{\prime}+qG_{n}=rG_{n-2}, \label{Gbiprime}
\end{gather}
and for an even $n$:
\begin{gather}
\widetilde{F}_{n}^{\prime}=f^{\prime}G_{n}-\rho\bigl(  f^{\prime}\widetilde
{G}_{n-1}-g^{\prime}\widetilde{F}_{n-1}\bigr), \label{Ftilprime}\\
\widetilde{G}_{n}^{\prime}=g^{\prime}F_{n}-\rho\bigl(  f^{\prime}\widetilde
{G}_{n-1}-g^{\prime}\widetilde{F}_{n-1}\bigr), \label{Gtilprime}\\
\bigl(  p\widetilde{F}_{n}^{\prime}\bigr)^{\prime}+q\widetilde{F}%
_{n}=r\widetilde{F}_{n-2}, \label{Ftilbiprime}\\
\bigl(  p\widetilde{G}_{n}^{\prime}\bigr)^{\prime}+q\widetilde{G}%
_{n}=r\widetilde{G}_{n-2}. \label{Gtilbiprime}%
\end{gather}
\end{lemma}

\begin{proof}
Let $n$ be odd. Then from (\ref{D2}), (\ref{D5}) and (\ref{D6}) we have
$F_{n}^{\prime}=\rho(  g^{\prime}F_{n-1}-f^{\prime}G_{n-1})  +\rho
fgr(  F_{n-2}-G_{n-2})  $. Due to (\ref{D2}) the difference in the
last brackets equals zero and hence (\ref{Fprime}) holds.

Consider $\left(  pF_{n}^{\prime}\right)  ^{\prime}=\rho\left(  \left(
pg^{\prime}\right)  ^{\prime}F_{n-1}-\left(  pf^{\prime}\right)  ^{\prime
}G_{n-1}+pr\left(  g^{\prime}fF_{n-2}-f^{\prime}gG_{n-2}\right)  \right)  $.
Now from (\ref{D2}), (\ref{W[f,g]}) and the fact that $f$ and $g$ are
solutions of (\ref{SL0}) we obtain $\left(  pF_{n}^{\prime}\right)  ^{\prime
}=-q\rho\left(  gF_{n-1}-fG_{n-1}\right)  +rF_{n-2}$ and hence (\ref{Fbiprime}%
). Equality (\ref{Gbiprime}) is proved similarly.

Let $n$ be even. Differentiating (\ref{D7}) and using (\ref{D3}), (\ref{D4})
and (\ref{D6}) we obtain%
\[
\widetilde{F}_{n}^{\prime}=f^{\prime}G_{n}+fgrG_{n-1}-\rho\bigl(  f^{\prime
}\widetilde{G}_{n-1}-g^{\prime}\widetilde{F}_{n-1}\bigr)  -\rho fgr\bigl(
\widetilde{G}_{n-2}-\widetilde{F}_{n-2}\bigr).
\]
Now using (\ref{D10}) we obtain (\ref{Ftilprime}). Equality (\ref{Gtilprime})
is proved analogously. Consider
\begin{align*}
\bigl(  p\widetilde{F}_{n}^{\prime}\bigr)^{\prime}  &  =(  pf^{\prime
})^{\prime}G_{n}+pf^{\prime}grG_{n-1}-\rho\bigl(\left(  pf^{\prime
}\right)  ^{\prime}\widetilde{G}_{n-1}-\left(  pg^{\prime}\right)  ^{\prime
}\widetilde{F}_{n-1}\bigr)-\rho\bigl(  pf^{\prime}\widetilde{G}_{n-1}^{\prime}-pg^{\prime}%
\widetilde{F}_{n-1}^{\prime}\bigr) \\
&  =-qfG_{n}+pf^{\prime}grG_{n-1}-q\rho\bigl(  g\widetilde{F}_{n-1}%
-f\widetilde{G}_{n-1}\bigr)  -\rho pr\bigl(  f^{\prime}g\widetilde{G}%
_{n-2}-g^{\prime}f\widetilde{F}_{n-2}\bigr) \\
&  =-q\widetilde{F}_{n}+\rho pr\bigl(  f^{\prime}g\bigl(  \widetilde{G}%
_{n-2}-\widetilde{F}_{n-2}\bigr)  -\bigl(  f^{\prime}g\widetilde{G}%
_{n-2}-g^{\prime}f\widetilde{F}_{n-2}\bigr)  \bigr) \\
&  =-q\widetilde{F}_{n}+r\widetilde{F}_{n-2}.
\end{align*}
Thus, (\ref{Ftilbiprime}) is true. Equality (\ref{Gtilbiprime}) is proved analogously.
\end{proof}

\begin{lemma}\label{Lemma Estimates}For the functions defined by Definition
\ref{Def Systems Fn} the following inequalities hold.
\begin{align}
&\left\vert F_{2k}(x)\right\vert \leq a_{2k}(c_{1}c_{2}c_{3})^{k}\left\vert
x-x_{0}\right\vert ^{k},& &\left\vert G_{2k}(x)\right\vert \leq a_{2k}%
(c_{1}c_{2}c_{3})^{k}\left\vert x-x_{0}\right\vert ^{k}, \label{F2k}\\
&\left\vert F_{2k+1}(x)\right\vert \leq a_{2k+1}c_{1}c_{3}(c_{1}c_{2}c_{3}
)^{k}\left\vert x-x_{0}\right\vert ^{k},& &\left\vert G_{2k+1}(x)\right\vert
\leq a_{2k+1}c_{1}c_{3}(c_{1}c_{2}c_{3})^{k}\left\vert x-x_{0}\right\vert
^{k}, \label{F2k+1}\\
&\vert \widetilde{F}_{2k}(x)\vert \leq b_{2k}c_{3}(c_{1}c_{2}%
c_{3})^{k}\left\vert x-x_{0}\right\vert ^{k},& &\vert \widetilde
{G}_{2k}(x)\vert \leq b_{2k}c_{3}(c_{1}c_{2}c_{3})^{k}\left\vert
x-x_{0}\right\vert ^{k}, \label{Ftil2k}\\
&\vert \widetilde{F}_{2k+1}(x)\vert \leq b_{2k+1}c_{2}c_{3}%
(c_{1}c_{2}c_{3})^{k}\left\vert x-x_{0}\right\vert ^{k+1},& &\vert
\widetilde{G}_{2k+1}(x)\vert \leq b_{2k+1}c_{2}c_{3}(c_{1}c_{2}%
c_{3})^{k}\left\vert x-x_{0}\right\vert ^{k+1}, \label{Ftil2k+1}%
\end{align}
where $c_{1}=\left\vert \rho\right\vert $, $c_{2}=\max\left(  \max
_{x\in\lbrack a,b]}\left\vert fr\right\vert ,\max_{x\in\lbrack a,b]}\left\vert
gr\right\vert \right)  $, $c_{3}=\max\left(  \max_{x\in\lbrack a,b]}\left\vert
f\right\vert ,\max_{x\in\lbrack a,b]}\left\vert g\right\vert \right)  $,
$a_{2k}=\frac{2^{k}}{k!}$, $a_{2k+1}=\frac{2^{k+1}}{k!}$, $b_{2k}=\frac
{2^{k}(k+1)}{k!}$, $b_{2k+1}=\frac{2^{k}}{k!}$, $k=0,1,\ldots$.
\end{lemma}

\begin{proof}
For $k=0$ all the inequalities are easily verified. Next, we assume that both
inequalities (\ref{F2k}) are true for some $k\in\mathbb{N}$ and consider
\begin{align*}
\left\vert F_{2k+1}(x)\right\vert  &  =\left\vert G_{2k+1}(x)\right\vert
=\left\vert \rho\left(  g(x)F_{2k}(x)-f(x)G_{2k}(x)\right)  \right\vert
\leq2a_{2k}c_{1}c_{3}(c_{1}c_{2}c_{3})^{k}\left\vert x-x_{0}\right\vert ^{k}\\
&  =a_{2k+1}c_{1}c_{3}(c_{1}c_{2}c_{3})^{k}\left\vert x-x_{0}\right\vert ^{k}.
\end{align*}
Hence
\[
\left\vert F_{2k+2}(x)\right\vert \leq\frac{a_{2k+1}}{k+1}(c_{1}c_{2}%
c_{3})^{k+1}\left\vert x-x_{0}\right\vert ^{k+1}=a_{2k+2}(c_{1}c_{2}%
c_{3})^{k+1}\left\vert x-x_{0}\right\vert ^{k+1}.
\]
Thus, (\ref{F2k}) and (\ref{F2k+1}) are proved.

Now, suppose that (\ref{Ftil2k}) hold for some $k\in\mathbb{N}$. Then
\[
\vert \widetilde{F}_{2k+1}(x)\vert \leq b_{2k}c_{2}c_{3}(c_{1}%
c_{2}c_{3})^{k}\frac{\left\vert x-x_{0}\right\vert ^{k+1}}{k+1}=b_{2k+1}%
c_{2}c_{3}(c_{1}c_{2}c_{3})^{k}\left\vert x-x_{0}\right\vert ^{k+1}.
\]
Consequently,
\begin{align*}
\vert \widetilde{F}_{2k+2}(x)\vert  &  =\left\vert f(x)G_{2k+2}%
(x)+\rho\bigl(  g(x)\widetilde{F}_{2k+1}(x)-f(x)\widetilde{G}_{2k+1}%
(x)\bigr)  \right\vert \\
&  \leq a_{2k+2}c_{3}(c_{1}c_{2}c_{3})^{k+1}\left\vert x-x_{0}\right\vert
^{k+1}+2b_{2k+1}c_{1}c_{2}c_{3}^{2}(c_{1}c_{2}c_{3})^{k}\left\vert
x-x_{0}\right\vert ^{k+1}.
\end{align*}
Notice that $b_{2k+2}=a_{2k+2}+2b_{2k+1}$ and hence $\vert \widetilde
{F}_{2k+2}(x)\vert \leq b_{2k+2}c_{3}(c_{1}c_{2}c_{3})^{k+1}\left\vert
x-x_{0}\right\vert ^{k+1}$. Thus, (\ref{Ftil2k}) and (\ref{Ftil2k+1}) are proved.
\end{proof}

Now we are in a position to prove the SPPS representations for solutions of
(\ref{SL}) in terms of the formal powers from Definition \ref{Def Systems Fn}.

\begin{theorem}[Modified SPPS representations]\label{Th Modified SPPS} Let $p$ and $q$ be
such that there exist two linearly independent solutions $f$ and $g$ of
equation \eqref{SL0} such that $\left\{  f,\,g,\,pf^{\prime},\,pg^{\prime
}\right\}  \subset C^{1}[a,b]$ and $f(x_{0})=g(x_{0})=1$ where $x_{0}$ is any
point of $[a,b]$ such that $p(x_{0})\neq0$. Let $r$ be such that $\left\{
fr,\,gr\right\}  \subset C[a,b]$. Then the general solution of \eqref{SL} on
$(a,b)$ has the form \eqref{genmain} where
\begin{equation}
u_{1}=\sum_{k=0}^{\infty}
\lambda^{k}\widetilde{F}_{2k}\quad\text{and}\quad u_{2}=\sum_{k=0}^{\infty}
\lambda^{k}F_{2k+1}. \label{u1u2}%
\end{equation}
The derivatives of $u_{1}$ and $u_{2}$ have the form%
\begin{equation}
pu_{1}^{\prime}=pf^{\prime}+\sum_{k=1}^{\infty}
\lambda^{k}\left(  pf^{\prime}G_{2k}-\rho\bigl(  pf^{\prime}\widetilde{G}%
_{2k-1}-pg^{\prime}\widetilde{F}_{2k-1}\bigr)  \right)  \label{u1prime}%
\end{equation}
and
\begin{equation}
pu_{2}^{\prime}=\rho\sum_{k=0}^{\infty}
\lambda^{k}\left(  pg^{\prime}F_{2k}-pf^{\prime}G_{2k}\right)  . \label{u2prime}%
\end{equation}
All series in \eqref{u1u2}--\eqref{u2prime} converge uniformly on $[a,b]$. The
solutions $u_{1}$ and $u_{2}$ satisfy the initial conditions
\begin{equation}
u_{1}(x_{0})=1,\quad u_{1}^{\prime}(x_{0})=f^{\prime}(x_{0}),\quad u_{2}%
(x_{0})=0,\quad u_{2}^{\prime}(x_{0})=\frac{1}{p(x_{0})}. \label{initial}%
\end{equation}
\end{theorem}

\begin{remark}\label{Rmk P in Der}
The function $p$ in \eqref{u1prime} and \eqref{u2prime} is necessary only in the case when this function possesses zeros and the derivatives $f'$ and $g'$ increase to infinity near the zeros of the function $p$. In all other cases we can easily remove all occurrences of $p$ in \eqref{u1prime} and \eqref{u2prime}.
\end{remark}

\begin{proof}
Lemma \ref{Lemma Estimates} guarantees the uniform convergence of all the
involved series. Moreover, it is not difficult to see that the majorizing
series for $\left\vert u_{1}(x)\right\vert $ converges to the function
$c_{3}\left(  1+c\left\vert x-x_{0}\right\vert \right)  e^{c\left\vert
x-x_{0}\right\vert }$ where $c=2\left\vert \lambda\right\vert c_{1}c_{2}c_{3}$
meanwhile the majorizing series corresponding to $\left\vert u_{2}%
(x)\right\vert $ converges to $2c_{1}c_{3}e^{c\left\vert x-x_{0}\right\vert }%
$. Indeed, we have
\[
\left\vert u_{1}(x)\right\vert \leq\sum_{k=0}^{\infty}
\left\vert \lambda\right\vert ^{k}\bigl\vert \widetilde{F}_{2k}(x)\bigr\vert
\leq c_{3}\sum_{k=0}^{\infty}
\left\vert \lambda\right\vert ^{k}\frac{2^{k}(k+1)}{k!}(c_{1}c_{2}c_{3}%
)^{k}\left\vert x-x_{0}\right\vert ^{k}.
\]
Observe that $\sum_{k=0}^{\infty}
\frac{(k+1)c^{k}}{k!}t^{k}=\left(  te^{ct}\right)  ^{\prime}=\left(
1+ct\right)  e^{ct}$. Hence%
\[
\left\vert u_{1}(x)\right\vert \leq c_{3}\left(  1+c\left\vert x-x_{0}%
\right\vert \right)  e^{c\left\vert x-x_{0}\right\vert }%
\]
where $c=2\left\vert \lambda\right\vert c_{1}c_{2}c_{3}$. Analogously we
have
\[
\left\vert u_{2}(x)\right\vert \leq2c_{1}c_{3}e^{c\left\vert x-x_{0}%
\right\vert }.
\]

Due to Lemma \ref{Lemma Derivatives} we obtain that $u_{1}$ and $u_{2}$ are
indeed solutions of (\ref{SL}) as well as the equalities (\ref{u1prime}) and
(\ref{u2prime}).

The equalities (\ref{initial}) follow from the fact that all formal powers
$F_{n}$, $G_{n}$, $\widetilde{F}_{n}$ and $\widetilde{G}_{n}$ vanish at
$x=x_{0}$ for any $n\in\mathbb{N}$. Finally, from (\ref{initial}) it follows
that $u_{1}$ and $u_{2}$ are linearly independent.
\end{proof}

\begin{remark}
The requirement to know two particular solutions of equation \eqref{SL0} as well as values of their derivatives at some point in Theorem \ref{Th Modified SPPS} does not present any difficulty for numerical applications, a variety of numerical methods can be used in order to construct two particular solutions, e.g., the SPPS representation can be successfully applied, see \cite{KrPorter2010}. Solely the case when only one particular solution is known exactly gives some advantage to the formulas \eqref{Xgen1}--\eqref{Xgen3}.
\end{remark}

\begin{remark}
The Modified SPPS representation presented in Theorem \ref{Th Modified SPPS}
works not only when particular solutions are available for $\lambda_0 = 0$, but in fact when two particular solutions of the equation $(pv')' + q v = \lambda_0 r v$ are known for some fixed $\lambda_0$. The solution \eqref{u1u2} now takes the form
\begin{equation}
u_{1}=\sum_{k=0}^{\infty}
(\lambda-\lambda_0)^{k}\widetilde{F}_{2k}\quad\text{and}\quad u_{2}=\sum_{k=0}^{\infty}
(\lambda-\lambda_0)^{k}F_{2k+1}. \label{u1u2SpShift}
\end{equation}
The procedure of using particular solutions at some point $\lambda_0\ne 0$ is called the spectral shift technique.
\end{remark}

\begin{remark}\label{RmkPiecewiseContinuousCoeffs}
The conditions $\left\{  f,\,g,\,pf^{\prime},\,pg^{\prime
}\right\}  \subset C^{1}[a,b]$ and $\left\{
fr,\,gr\right\}  \subset C[a,b]$ in Theorem \ref{Th Modified SPPS} are superfluous and are necessarily only if we are interested in the classical solutions of equation \eqref{SL}. If we allow weak solutions, the SPPS representations of the general solution (both original and modified) can be obtained under weaker assumptions on the coefficients, namely when $\left\{  f,\,g,\,pf^{\prime},\,pg^{\prime
}\right\}  \subset AC[a,b]$ and $\left\{
fr,\,gr\right\}  \subset L^1[a,b]$. We refer the reader to \cite{BCK} for further details.
\end{remark}

Since the formal powers are the essential ingredient of several methods for
solving equations and corresponding spectral problems it is important to verify whether the method of their calculation based on two particular solutions (Definition \ref{Def Systems Fn}), we will call it the new method, presents computational advantages in comparison to the direct recursive integration (formulas (\ref{Xgen1})--(\ref{Xgen3})), the old method. It is clear that the new method of construction of the formal powers is applicable even when the function $1/(f^{2}p)$ is not necessarily continuous on $[a,b]$. For example, $f$ and $p$ can possess zeros on $[a,b]$. This is an important extension of applicability of the SPPS approach. Apart from it, we can highlight the following computational advantages of the new method.

\begin{enumerate}
\item The first several formal powers (whose contribution in the final result usually is greater than that of subsequent formal powers) are computed with a higher accuracy.
\item More formal powers can be computed. See for details \cite[Examples 7.3 and 7.7]{KrT2013}.
\item Computation of formal powers is considerably more stable, especially
when the particular solution $f$ is of a larger change or nearly vanishing on the interval of interest.
\item Computation of the formal powers by the new method requires the same number of integrations as by the old method and only several more algebraic operations, i.e., the computation time essentially does not increase. In some cases the new method may be several times faster than the old one, this is due to the necessity to use complex-valued particular solution for the old method to ensure that this solution does not vanish, meanwhile for the new method one still can work with real-valued particular solutions.
\item Accuracy is much higher when the particular solution $f$ or/and the
coefficient $p$ possess values close to zero on $[a,b]$.
\end{enumerate}

Below we illustrate these points.

\begin{example}
Consider the function $f(x)=1+cx$ which is obviously a particular solution of
the equation $f^{\prime\prime}(x)=0$ and $f(0)=1$. As a second particular
solution of the same equation satisfying the condition $g(0)=1$ we can choose
the function $g\equiv1$. The corresponding formal powers will be considered on
the segment $[0,10]$. It is easy to see that $G_{n}(x)=x^{n}/n!$. Moreover,
due to \eqref{D2} we have that for an odd $n$: $F_{n}(x)=x^{n}/n!$ meanwhile
for an even $n$ the formal powers $F_{n}$ have the form $F_{n}(x)=\frac{x^{n}%
}{(n+1)!}\left(  n\left(  1+cx\right)  +1\right)  $. In a similar way the
formal powers $\widetilde{F}_{n}$ for this example can be written down
explicitly by means of Definition \ref{Def Systems Fn}. All the calculations
of the recursive integrals were performed in Matlab using the Newton-Cottes 6
point integration formula of 7-th order (see, e.g., \cite{Rabinowitz}) with
$10^{5}$ uniformly distributed nodes. In all cases the computation took
several seconds. The presented numerical results correspond to odd $n$, and
the figures show the following difference $\left\vert x^{n}-n!F_{n}%
(x)\right\vert /\max_{[0,10]}x^{n}=\left\vert x^{n}-n!F_{n}(x)\right\vert
/10^{n}$.

First, we consider a case when $f$ is a nice function: $c=1$. The first few formal powers are computed more accurately by the new method meanwhile for the higher formal powers the old method resulted to be preferable. Nevertheless even in this ``nice'' case the error produced by the new method is not much worse than the error of the old method, see Fig. \ref{Fig Accuracy F} (a).

\begin{figure}[htb]
\centering
\begin{tabular}{cc}
\includegraphics[
height=2in,
width=3in,
bb=198 324 414 468
]
{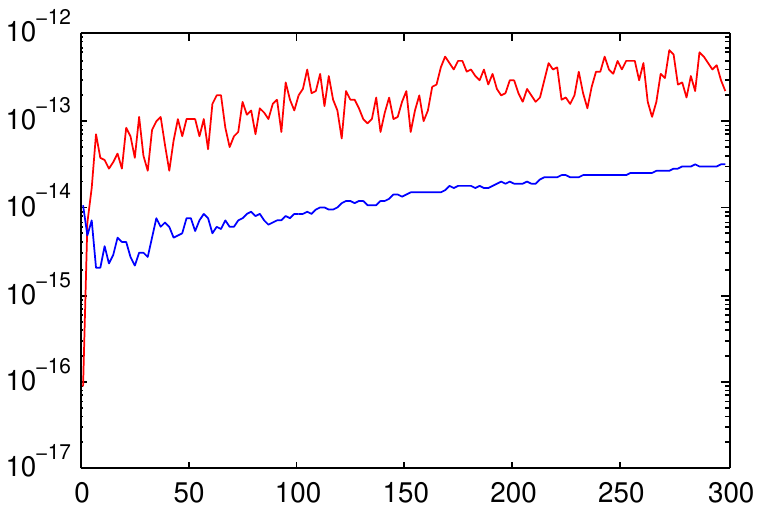}
&
 \includegraphics[
height=2in,
width=3in,
bb=198 324 414 468
]
{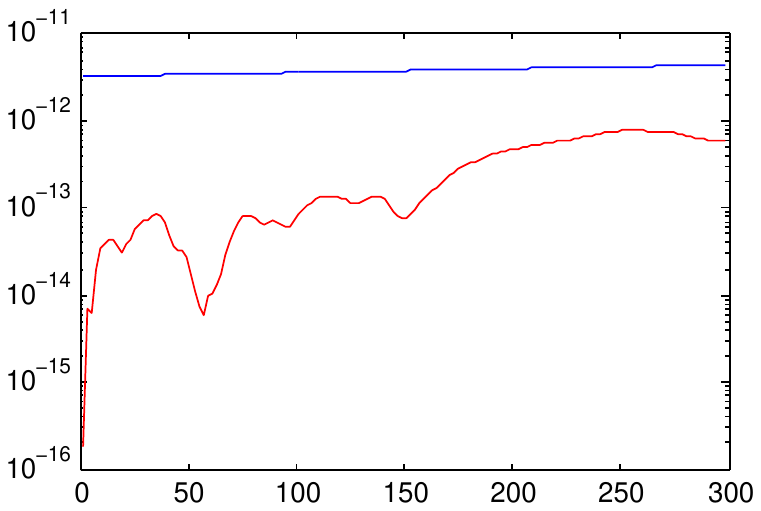}\\
(a) & (b) \\
\includegraphics[
height=2in,
width=3in,
bb=198 324 414 468
]
{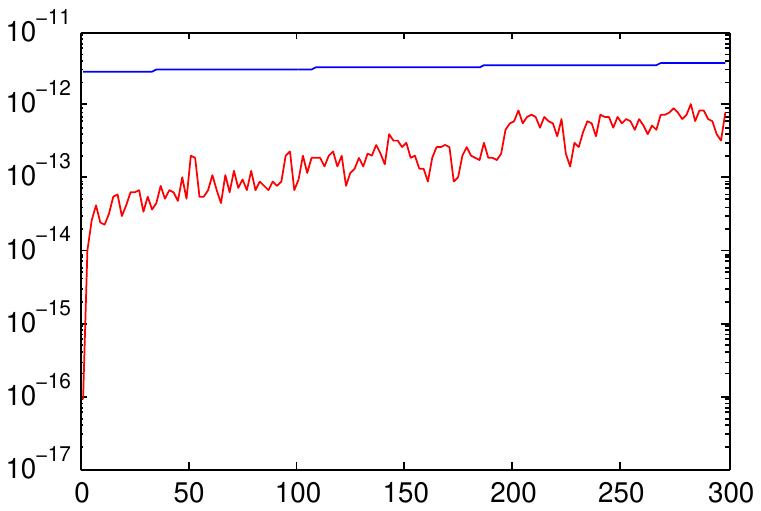}
&
\includegraphics[
height=2in,
width=3in,
bb=198 324 414 468
]
{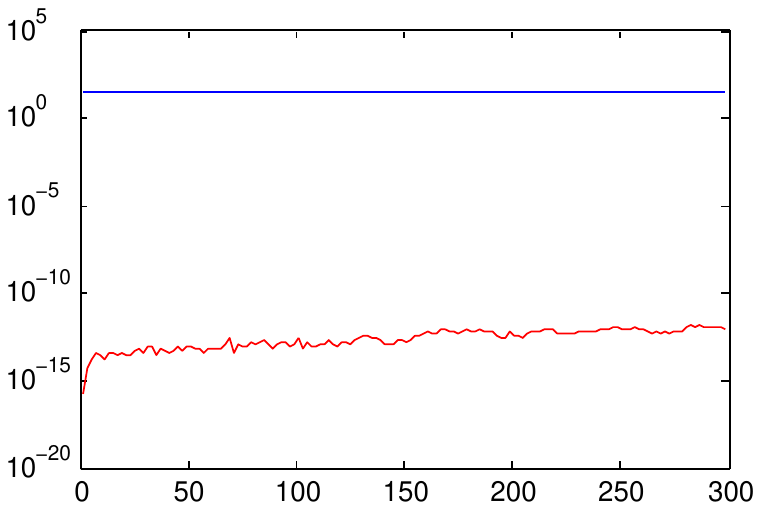}\\
(c) & (d)
\end{tabular}
\caption{The blue line (which starts above) shows the error of the formal
powers $F_{n}$, for odd $n$ computed by the old method. The red line
(starts below) shows the same but computed by the new method. The following values of the parameter $c$ are used:
(a)  $c=1$, (b) $c=0.0001-1/10$, (c) $c=100$ and (d) $c=1000000$.}
\label{Fig Accuracy F}
\end{figure}

Fig. \ref{Fig Accuracy F} (b) shows that the accuracy
achieved in the case of an almost vanishing function $f$ (here $c=0.0001-1/10$) is considerably better when the new method is applied.

Taking $c=100$ one can observe on Fig. \ref{Fig Accuracy F} (c) that the situation with the accuracy changes considerably for the old method meanwhile the new method delivers similar results as on Fig. \ref{Fig Accuracy F} (a). Moreover, further increasing $c$ and hence making the function $f$ take larger
values we easily arrive at a situation when the old method becomes practically
useless meanwhile the new method keeps delivering accurate results. Fig.
\ref{Fig Accuracy F} (d) corresponds to $c=1000000$.
\end{example}

\subsection{General solution in terms of the formal powers for Darboux associated equations}\label{subsect Darboux}
Suppose that $f$ and $g$ are nonvanishing on a segment of interest $[a,b]$
linearly independent solutions of (\ref{SL0}) such that $f(x_{0})=g(x_{0})=1$,
$x_{0}\in\lbrack a,b]$. Then together with equation (\ref{SL}) let us consider
the following Sturm-Liouville equations%
\begin{equation}
\left(  \frac{1}{r}v^{\prime}\right)  ^{\prime}+q_{1/f}v=\lambda\frac{1}%
{p}v\label{SLv}%
\end{equation}
and
\begin{equation}
\left(  \frac{1}{r}w^{\prime}\right)  ^{\prime}+q_{1/g}w=\lambda\frac{1}%
{p}w\label{SLw}%
\end{equation}
where%
\[
q_{1/f}=-\left(  \frac{q}{pr}+\frac{2}{r}\left(  \frac{f^{\prime}}{f}\right)
^{2}+\frac{f^{\prime}}{fr}\frac{\left(  pr\right)  ^{\prime}}{pr}\right)
\]
and $q_{1/g}$ has the same form as $q_{1/f}$ with $f$ being replaced
everywhere by $g$.

The functions $1/f$ and $1/g$ are solutions of (\ref{SLv}) and (\ref{SLw}) corresponding to $\lambda=0$
respectively. We will call (\ref{SLv}) and (\ref{SLw}) the Sturm-Liouville
equations Darboux associated with (\ref{SL}).

Let us observe that the functions
\[
v_{1}=\frac{1}{f}\sum_{k=0}^{\infty}
\lambda^{k}X^{(2k)}\qquad\text{and}\qquad v_{2}=\frac{1}{f}\sum_{k=0}^{\infty}
\lambda^{k}\widetilde{X}^{(2k+1)}
\]
are linearly independent solutions of \eqref{SLv} as well as the functions
\[
w_{1}=\frac{1}{g}\sum_{k=0}^{\infty}
\lambda^{k}Y^{(2k)}\qquad\text{and}\qquad w_{2}=\frac{1}{g}\sum_{k=0}^{\infty}
\lambda^{k}\widetilde{Y}^{(2k+1)}
\]
are linearly independent solutions of \eqref{SLw}.

Now, from \eqref{u1u2} and \eqref{D7} we have that
\begin{equation}
u_{1}=fg\left(  w_{1}-\lambda\rho\left(  w_{2}-v_{2}\right)  \right)
,\label{u1viaDarboux}%
\end{equation}
and from \eqref{u1u2} and \eqref{D2},
\begin{equation}
u_{2}=\rho fg\left(  v_{1}-w_{1}\right)  .\label{u2viaDarboux}%
\end{equation}
Equalities \eqref{u1viaDarboux} and \eqref{u2viaDarboux} give us expressions
for the solutions of \eqref{SL} in terms of solutions of the
Darboux-associated equations \eqref{SLv} and \eqref{SLw}.

\begin{remark}\label{Rmk Darboux}
The observation that for a Darboux-associated equation one has to calculate
the same formal powers as for the original Sturm-Liouville equation can be
used in the following way. Suppose that $1/p$ is a ``nice'' function meanwhile $r$ is ``nasty'', e.g., has a singularity or even an ``almost'' singularity, achieving very large values. In this
case one might prefer to calculate the integrals containing $1/(fp)$ in the
integrand rather than those containing $fr$. For this it is sufficient to
consider equation \eqref{SLv} and follow the described above construction
begining with Definition \ref{Def Systems Fn} where now the roles of $p$ and
$r$ result to be interchanged.
\end{remark}

\section{SPPS representations for solutions of pencils of Sturm-Liouville
operators}\label{Section3}
In this section we show that the SPPS representations analogous to those
established in Theorem \ref{Th Modified SPPS} can also be obtained for
solutions of Sturm-Liouville equations of the form
\begin{equation}
(p(x)u')'+q(x)u=\sum_{k=1}^{N}
\lambda^{k}R_{k}\left[  u\right]  ,\qquad x\in(a,b) \label{pencil}%
\end{equation}
where $R_{k}$ are linear differential operators of the first order,
$R_{k}\left[  u\right]  :=r_{k}(x)u+s_{k}(x)u^{\prime}$, $k=1,\ldots N$, the
complex-valued functions $p$, $q$, $r_{k}$, $s_{k}$ are continuous on the
finite segment $\left[  a,b\right]  $.

\subsection{SPPS representation for solutions of pencils}
It is possible to obtain the general solution of equation \eqref{pencil} by slightly changing the definition of formal powers (\ref{Xgen1})--(\ref{Xgen3}). We define the formal powers for equation \eqref{pencil} as follows
\begin{align}
\widetilde{\mathcal{X}}^{\left(  -n\right)  }&\equiv\mathcal{X}^{\left(
-n\right)  }\equiv0\qquad\text{ for }n\in\mathbb{N},\label{Xpencil1}\\
\widetilde{\mathcal{X}}^{\left(  0\right)  }&\equiv\mathcal{X}^{\left(  0\right)  }\equiv1,\\
\displaybreak[2]
\widetilde{\mathcal{X}}^{\left(  n\right)  }(x)   &  =
\begin{cases}\displaystyle\int_{x_{0}}^{x}
f(s)\sum_{k=1}^{N}R_{k}\left[  f(s)\widetilde{\mathcal{X}}^{\left(
n-2k+1\right)  }(s)\right]  ds, & n\text{ - odd,}\smallskip \\
\displaystyle\int_{x_{0}}^{x}
\widetilde{\mathcal{X}}^{\left(  n-1\right)  }\left(  s\right)  \dfrac
{ds}{f^{2}\left(  s\right)  p\left(  s\right)  }, & n\text{ - even,}
\end{cases}
\label{eq: equis tilde}\\
\mathcal{X}^{\left(  n\right)  }(  x)   &  =
\begin{cases}\displaystyle\int_{x_{0}}^{x}
\mathcal{X}^{\left(  n-1\right)  }\left(  s\right)  \dfrac{ds}{f^{2}\left(
s\right)  p\left(  s\right)  }, & n\text{ - odd,}\smallskip  \\
\displaystyle\int_{x_{0}}^{x}
f(s)\sum_{k=1}^{N}R_{k}\left[  f(s)\mathcal{X}^{\left(  n-2k+1\right)
}(s)\right]  ds, & n\text{ - even}%
\end{cases}
\label{eq: equis sin tilde}%
\end{align}
where $x_{0}$ is an arbitrary point of the segment $\left[  a,b\right]  $ such
that $p(x_{0})\neq0$. The following theorem generalizes Theorem \ref{ThSolSL}.

\begin{theorem}[SPPS representations for polynomial pencils of operators]
\label{ThSPPS_Pencil} Assume that on a finite interval $[a,b]$, equation
\eqref{SL0} possesses a particular solution $f$ such that the functions
$fR_{k}[f]$, $k=1,\ldots,N$ and $\frac{1}{f^{2}p}$ are continuous on $\left[
a,b\right]  $. Then the general solution of \eqref{pencil} has the form
$u=c_{1}u_{1}+c_{2}u_{2}$, where $c_{1}$ and $c_{2}$ are arbitrary complex
constants and
\begin{equation}
u_{1}=f\sum_{n=0}^{\infty}\lambda^{n}\widetilde{\mathcal{X}}^{\left(
2n\right)  }\qquad\text{and}\qquad u_{2}=f\sum_{n=0}^{\infty}\lambda
^{n}\mathcal{X}^{\left(  2n+1\right)  }.\label{solPencil}%
\end{equation}
Both series in \eqref{solPencil} converge uniformly on $\left[  a,b\right]$.
\end{theorem}

The formulation and the proof of this theorem in the case $s_{k}\equiv0$, $k=1,\ldots,N$ can be found in \cite{KTV}. An analogous theorem for a perturbed Bessel equation in the case $N=1$ can be found in \cite{CKT2013}. The proof from \cite{KTV} can be easily generalized onto the case considered
here. Nevertheless we do not present here the proof of Theorem \ref{ThSPPS_Pencil} because below we
prove a stronger result generalizing Theorem \ref{Th Modified SPPS} and allowing particular solution to have zeros.

\subsection{Modified SPPS representation for solutions of pencils}
We introduce the following definition (cf. Definition \ref{Def Systems Fn})
where in order not to overload this paper with additional notations we use the
same characters as above.

\begin{definition}
\label{Def Powers Pencil}Let equation \eqref{SL0} admit two linearly
independent solutions $f$ and $g$ such that $\left\{  f,\,g,\,pf^{\prime
},\,pg^{\prime}\right\}  \subset C^{1}[a,b]$ and $f(x_{0})=g(x_{0})=1$ where
$x_{0}$ is any point of $[a,b]$ such that $p(x_{0})\neq0$. Then the following
systems of functions $\{  F_{n}\}  $, $\{  \widetilde{F}_{n}\}$,
$\{  G_{n}\}$, $\{\widetilde{G}_{n}\}$ are defined recursively as follows
\begin{align}
F_{-n}&\equiv G_{-n}\equiv\widetilde{F}_{-n}\equiv\widetilde{G}_{-n}%
\equiv0\qquad\text{for }n\in\mathbb{N},\label{FGnegPencil}\\
F_{0}&\equiv G_{0}\equiv1, \qquad\widetilde{F}_{0}\equiv f, \qquad
\widetilde{G}_{0}\equiv g,\nonumber
\end{align}
for an odd $n$:
\begin{align}
F_{n}&=G_{n}=\rho\left(  gF_{n-1}-fG_{n-1}\right), \nonumber\\
\widetilde{F}_{n}(x)&=
\int_{x_{0}}^{x}
f(s)\sum_{k=1}^{N}R_{k}\left[  \widetilde{F}_{n-2k+1}(s)\right]  \,ds,
\label{FtilnPencil}\\
\widetilde{G}_{n}(x)&=
\int_{x_{0}}^{x}
g(s)\sum_{k=1}^{N}R_{k}\left[  \widetilde{G}_{n-2k+1}(s)\right]  \,ds,
\label{GtilnPencil}%
\end{align}
and for an even $n$:
\begin{align}
F_{n}(x)&=\int_{x_{0}}^{x}
f(s)\sum_{k=1}^{N}R_{k}\left[  F_{n-2k+1}(s)\right]  \,ds, \label{FnPencil}\\
G_{n}(x)&=\int_{x_{0}}^{x}
g(s)\sum_{k=1}^{N}R_{k}\left[  G_{n-2k+1}(s)\right]  \,ds, \label{GnPencil}\\
\widetilde{F}_{n}&=fG_{n}-\rho\left(  f\widetilde{G}_{n-1}-g\widetilde{F}_{n-1}\right),\nonumber\\
\widetilde{G}_{n}&=gF_{n}-\rho\left(  f\widetilde{G}_{n-1}-g\widetilde{F}_{n-1}\right).\nonumber
\end{align}
\end{definition}

From the last two equalities we have
\[
\widetilde{G}_{2n}-\widetilde{F}_{2n}=gF_{2n}-fG_{2n}.
\]
This definition may give an impression that the calculation of the formal
powers involves their differentiation (application of the operators $R_{k}$
under the sign of integral). Nevertheless it is easy to see that such
differentiation is superfluous. Namely, we have the following equalities for
the $F$-formal powers
\begin{align}
R_{k}\left[  F_{2n+1}\right]  &=\rho\left(  R_{k}\left[  g\right]  F_{2n}%
-R_{k}\left[  f\right]  G_{2n}\right)  , \label{RkF2n+1}\\
R_{k}[\widetilde{F}_{2n}]  &=R_{k}\left[  f\right]  G_{2n}%
+\rho\bigl(  R_{k}\left[  g\right]  \widetilde{F}_{2n-1}-R_{k}\left[
f\right]  \widetilde{G}_{2n-1}\bigr)  \label{RkFtil}%
\end{align}
as well as analogous equalities for the $G$-formal powers $G_{2n+1}$ and
$\widetilde{G}_{2n}$ with obvious substitution of $f$ by $g$ and vice versa.
For the proof of (\ref{RkF2n+1}) it is sufficient to observe that
$gF_{2n}^{\prime}-fG_{2n}^{\prime}=0$. Indeed,
\[
gF_{2n}^{\prime}-fG_{2n}^{\prime}=fg\left(  \sum_{k=1}^{N}R_{k}\left[
F_{2n-2k+1}\right]  -\sum_{k=1}^{N}R_{k}\left[  G_{2n-2k+1}\right]  \right)
\]
which equals zero because every operator $R_{k}$ is linear and $F_{2n-2k+1}%
\equiv G_{2n-2k+1}$ by definition. Equality (\ref{RkFtil}) is proved in a
similar way.

Thus, for a practical use of Definition \ref{Def Powers Pencil} instead of
(\ref{FtilnPencil}) and (\ref{GtilnPencil}) it is convenient to use an
alternative form of these equalities which does not require differentiation of
formal powers%
\begin{align}
\widetilde{F}_{2n+1}(x)    &=\int_{x_{0}}^{x}
f(s)\sum_{k=1}^{N}\left(  R_{k}\left[  f(s)\right]  G_{2n-2k+2}(s)+\rho\bigl(
R_{k}\left[  g(s)\right]  \widetilde{F}_{2n-2k+1}(s)-R_{k}\left[  f(s)\right]
\widetilde{G}_{2n-2k+1}(s)\bigr)  \right)  \,ds,\label{FtilnPencilAlt}\\
\widetilde{G}_{2n+1}(x)    &=\int_{x_{0}}^{x}
g(s)\sum_{k=1}^{N}\left(  R_{k}\left[  g(s)\right]  F_{2n-2k+2}(s)+\rho\bigl(
R_{k}\left[  g(s)\right]  \widetilde{F}_{2n-2k+1}(s)-R_{k}\left[  f(s)\right]
\widetilde{G}_{2n-2k+1}(s)\bigr)  \right)  \,ds,\label{GtilnPencilAlt}
\end{align}
and analogously, instead of (\ref{FnPencil}) and (\ref{GnPencil}) their
alternative form
\begin{align}
F_{2n}(x)&=\rho\int_{x_{0}}^{x}
f(s)\sum_{k=1}^{N}\left(  R_{k}\left[  g(s)\right]  F_{2n-2k}(s)-R_{k}\left[
f(s)\right]  G_{2n-2k}(s)\right)  \,ds,
\label{FnPencilAlt}\\
G_{2n}(x)&=\rho\int_{x_{0}}^{x}
g(s)\sum_{k=1}^{N}\left(  R_{k}\left[  g(s)\right]  F_{2n-2k}(s)-R_{k}\left[
f(s)\right]  G_{2n-2k}(s)\right)  \,ds.
\label{GnPencilAlt}%
\end{align}

\begin{lemma}
\label{Lemma Derivatives Pencil}For the functions defined by Definition
\ref{Def Powers Pencil} the following relations hold.

For an odd $n$:%
\begin{gather}
F_{n}'=G_{n}'=\rho\left(  g^{\prime}F_{n-1}-f^{\prime}%
G_{n-1}\right),\nonumber\\
(pF_{n}')'+qF_{n}=\sum_{k=1}^{N}R_{k}\left[
F_{n-2k}\right]  , \label{pFnprimePencil}\\
(pG_{n}')'+qG_{n}=\sum_{k=1}^{N}R_{k}\left[
G_{n-2k}\right],\nonumber
\end{gather}
and for an even $n$:
\begin{gather*}
\widetilde{F}_{n}^{\prime}=f^{\prime}G_{n}-\rho\bigl(  f^{\prime}\widetilde
{G}_{n-1}-g^{\prime}\widetilde{F}_{n-1}\bigr),\\
\widetilde{G}_{n}^{\prime}=g^{\prime}F_{n}-\rho\bigl(  f^{\prime}\widetilde
{G}_{n-1}-g^{\prime}\widetilde{F}_{n-1}\bigr),\\
(p\widetilde{F}_{n}')'+q\widetilde{F}_{n}%
=\sum_{k=1}^{N}R_{k}\bigl[  \widetilde{F}_{n-2k}\bigr],\\
(p\widetilde{G}_{n}')'+q\widetilde{G}_{n}
=\sum_{k=1}^{N}R_{k}\bigl[  \widetilde{G}_{n-2k}\bigr].
\end{gather*}

\end{lemma}

\begin{proof}
The proof of the equalities for the first derivatives of the formal powers is
completely analogous to that from Lemma \ref{Lemma Derivatives}. We will prove
(\ref{pFnprimePencil}), the rest of the equalities involving second
derivatives of the formal powers are proved similarly. Consider
\begin{equation}
(pF_{n}')'=\rho\bigl(  (pg')' F_{n-1}-(pf')'
G_{n-1}\bigr)  +\rho p\left(  g^{\prime}F_{n-1}^{\prime}-f^{\prime}%
G_{n-1}^{\prime}\right)  . \label{vsp1}%
\end{equation}
Since $p\left(  g^{\prime}F_{n-1}^{\prime}-f^{\prime}G_{n-1}^{\prime}\right)
=p\left(  g^{\prime}f-f^{\prime}g\right)  \sum_{k=1}^{N}R_{k}\left[
F_{n-2k}\right]  =\frac{1}{\rho}\sum_{k=1}^{N}R_{k}\left[  F_{n-2k}\right]  $,
from (\ref{vsp1}) we have
\[
(pF_{n}')'=-\rho q\left(  gF_{n-1}%
-fG_{n-1}\right)  +\sum_{k=1}^{N}R_{k}\left[  F_{n-2k}\right]
\]
which is (\ref{pFnprimePencil}).
\end{proof}

\begin{lemma}
\label{Lemma Estimates Pencil}Let
$c_{1}=\left\vert \rho\right\vert $, $c_{2}=\max_{k=\overline{1,N}}(  \max_{x\in[a,b]}\left\vert
R_{k}[f]\right\vert ,\max_{x\in[a,b]}\left\vert R_{k}[g]\right\vert)$ and $c_{3}=\max(  \max_{x\in[a,b]}\left\vert f\right\vert,$ $\max
_{x\in[a,b]}\left\vert g\right\vert )$.
Then for the functions
defined by Definition \ref{Def Powers Pencil} the following inequalities
hold.
\begin{align}
\left\vert F_{2n}(x)\right\vert &\leq
\sum_{k=0}^{n-\left[  \frac{n}{N}\right]}
\binom{n}{k}\frac{(2c_{1}c_{2}c_{3})^{n-k}\left\vert x-x_{0}\right\vert
^{n-k}}{\left(  n-k\right)  !},\label{F2nPencil}\\\displaybreak[2]
\left\vert G_{2n}(x)\right\vert &\leq
\sum_{k=0}^{n-\left[  \frac{n}{N}\right]  }
\binom{n}{k}\frac{(2c_{1}c_{2}c_{3})^{n-k}\left\vert x-x_{0}\right\vert
^{n-k}}{\left(  n-k\right)  !} \label{G2nPencil}\\\displaybreak[2]
\left\vert F_{2n+1}(x)\right\vert =\left\vert G_{2n+1}(x)\right\vert
&\leq2c_{1}c_{3}\sum_{k=0}^{n-\left[  \frac{n}{N}\right]  }
\binom{n}{k}\frac{(2c_{1}c_{2}c_{3})^{n-k}\left\vert x-x_{0}\right\vert
^{n-k}}{\left(  n-k\right)  !}, \label{F2n+1Pencil}\\\displaybreak[2]
\left\vert \widetilde{F}_{2n}(x)\right\vert &\leq c_{3}\sum_{k=0}^{n-\left[  \frac{n}{N}\right]  }
\binom{n}{k}\frac{(2c_{1}c_{2}c_{3})^{n-k}\left\vert x-x_{0}\right\vert
^{n-k}(n-k+1)}{\left(  n-k\right)  !},\label{Ftil2nPencil}\\\displaybreak[2]
\left\vert \widetilde{G}_{2n}%
(x)\right\vert &\leq c_{3}\sum\limits_{k=0}^{n-\left[  \frac{n}{N}\right]  }
\binom{n}{k}\frac{(2c_{1}c_{2}c_{3})^{n-k}\left\vert x-x_{0}\right\vert
^{n-k}(n-k+1)}{\left(  n-k\right)  !}, \label{Gtil2nPencil}\\\displaybreak[2]
\left\vert \widetilde{F}_{2n+1}(x)\right\vert &\leq c_{2}c_{3}
\sum_{k=0}^{n+1-\left[  \frac{n+1}{N}\right]  }
\binom{n+1}{k}\frac{(2c_{1}c_{2}c_{3})^{n-k}\left\vert x-x_{0}\right\vert
^{n+1-k}}{\left(  n-k\right)  !}, \label{Ftil2n+1Pencil}\\
\left\vert \widetilde{G}_{2n+1}(x)\right\vert &\leq c_{2}c_{3}
\sum_{k=0}^{n+1-\left[  \frac{n+1}{N}\right]  }
\binom{n+1}{k}\frac{(2c_{1}c_{2}c_{3})^{n-k}\left\vert x-x_{0}\right\vert
^{n+1-k}}{\left(  n-k\right)  !}, \label{Gtil2n+1Pencil}
\end{align}
where $[x]$ denotes the largest integer less than or equal to $x$.
\end{lemma}

\begin{remark}
In the case when $N=1$ and $s_1\equiv 0$, the estimates \eqref{F2nPencil}--\eqref{Gtil2n+1Pencil} coincide with the estimates given in Lemma \ref{Lemma Estimates}.
\end{remark}

\begin{proof}
Clearly inequalities \eqref{F2nPencil} and \eqref{G2nPencil} hold for $n=0$. Assume that inequalities \eqref{F2nPencil} and \eqref{G2nPencil} hold for all $n$, $0\le n < m$ for some $m\in\mathbb{N}$. Then taking into account \eqref{FGnegPencil} we obtain from \eqref{FnPencilAlt} that
\begin{align*}
    |F_{2m}(x)|&=\biggl|\rho\int_{x_0}^x f(s)\sum_{j=1}^{\min(N,m)}\left(  R_{j}\left[  g(s)\right]  F_{2m-2j}(s)-R_{j}\left[
f(s)\right]  G_{2m-2j}(s)\right)ds\biggr|\\
&\le 2c_1c_2c_3\sum_{j=1}^{\min(N,m)}\int_{x_0}^x \sum_{k=0}^{m-j-\left[  \frac{m-j}{N}\right]}
\binom{m-j}{k}\frac{(2c_{1}c_{2}c_{3})^{m-j-k}\left\vert x-x_{0}\right\vert
^{m-j-k}}{\left(  m-j-k\right)  !}ds\\
&=\sum_{j=1}^{\min(N,m)}\sum_{k=0}^{m-j-\left[  \frac{m-j}{N}\right]}
\binom{m-j}{k}\frac{(2c_{1}c_{2}c_{3})^{m-j-k+1}\left\vert x-x_{0}\right\vert
^{m-j-k+1}}{\left(  m-j-k+1\right)  !}.
\end{align*}
We rearrange the terms with respect to $\ell = k+j-1$. It follows from $1\le j\le \min(N,m)$ and $0\le k\le m-j-\bigl[\frac{m-j}N\bigr]$ that $0\le \ell\le m-1-\bigl[\frac{m-j}N\bigr]\le m-1-\bigl[\frac{m-N}N\bigr] = m-\bigl[\frac{m}N\bigr]$ and that $j\le \min(N,m,\ell+1)$. Hence
\begin{align*}
    |F_{2m}(x)|&\le \sum_{\ell = 0}^{m-\left[\frac{m}{N}\right]}\sum_{j=1}^{\min(N,m,\ell+1)}
    \binom{(m-1)-(j-1)}{\ell - (j-1)}\frac{(2c_{1}c_{2}c_{3})^{m-\ell}\left\vert x-x_{0}\right\vert
^{m-\ell}}{\left(  m-\ell\right)  !}\\
&\le \sum_{\ell = 0}^{m-\left[\frac{m}{N}\right]}
    \frac{(2c_{1}c_{2}c_{3})^{m-\ell}\left\vert x-x_{0}\right\vert
^{m-\ell}}{\left(  m-\ell\right)  !}\sum_{j=0}^{\ell}\binom{m-1-j}{\ell - j}=
\sum_{\ell = 0}^{m-\left[\frac{m}{N}\right]}\binom{m}{\ell}
    \frac{(2c_{1}c_{2}c_{3})^{m-\ell}\left\vert x-x_{0}\right\vert
^{m-\ell}}{\left(  m-\ell\right)  !}.
\end{align*}
Similarly we obtain inequality \eqref{G2nPencil}. Now \eqref{F2n+1Pencil} easily follows from the definition.

It is easy to see from \eqref{FtilnPencilAlt}, \eqref{GtilnPencilAlt} that inequalities \eqref{Ftil2n+1Pencil} and \eqref{Gtil2n+1Pencil} hold for $n=0$. Assume that inequalities \eqref{Ftil2n+1Pencil} and \eqref{Gtil2n+1Pencil} hold for all $n$, $0\le n<m$. Similarly to the first part of the proof we obtain from \eqref{FtilnPencilAlt} that
\begin{align*}
    |\widetilde F_{2m+1}(x)|& \le \sum_{j=1}^{\min(N,m+1)}c_2c_3\sum_{k=0}^{m+1-j-\left[  \frac{m+1-j}{N}\right]}
\binom{m+1-j}{k}\frac{(2c_{1}c_{2}c_{3})^{m-j-k+1}\left\vert x-x_{0}\right\vert
^{m-j-k+2}}{\left(  m-j-k+2\right)  !}\\
& \qquad +
\sum_{j=1}^{\min(N,m)}2c_1c_2c_3\cdot c_2c_3\sum_{k=0}^{m+1-j-\left[  \frac{m+1-j}{N}\right]}
\binom{m+1-j}{k}\frac{(2c_{1}c_{2}c_{3})^{m-j-k}\left\vert x-x_{0}\right\vert
^{m-j-k+2}}{(m-j-k+2)\cdot\left(  m-j-k\right)  !}\\
&\le c_2c_3\sum_{j=1}^{\min(N,m+1)}\sum_{k=0}^{m+1-j-\left[  \frac{m+1-j}{N}\right]}
\binom{m+1-j}{k}\frac{(2c_{1}c_{2}c_{3})^{m-j-k+1}\left\vert x-x_{0}\right\vert
^{m-j-k+2}}{\left(  m-j-k+1\right)  !},
\end{align*}
end the proof can be finished as in the first part.

Now inequalities \eqref{Ftil2nPencil} and \eqref{Gtil2nPencil} easily follow from the definition.
\end{proof}

The following corollary presents rougher estimates than those in Lemma \ref{Lemma Estimates Pencil} however better suited for the convergency testing.
\begin{corollary}\label{Corr Estimates Pencil}
Under the conditions of Lemma \ref{Lemma Estimates Pencil} define
\begin{equation*}
    C(n,x):=\frac{(1+2c_1c_2c_3|x-x_0|)^n}{\bigl(\bigl[\frac nN\bigr]\bigr)!}.
\end{equation*}
Then for the functions $F_n$, $\widetilde F_n$, $n\ge 0$, the following estimates hold.
\begin{align*}
    |F_{2n}(x)|&\le C(n,x), & |F_{2n+1}(x)| &\le 2c_1c_3 C(n,x),\\
    |\widetilde F_{2n}(x)|&\le (n+1)c_3  C(n,x), & |\widetilde F_{2n+1}(x)| &\le \frac{n+1}{2c_1} C(n+1,x).
\end{align*}
The same estimates hold for the functions $G_n$, $\widetilde G_n$.
\end{corollary}

\begin{proof}
Consider the inequality \eqref{Ftil2n+1Pencil}. We have $n+1-k\ge \bigl[\frac {n+1}N\bigr]$ hence
\begin{align*}
    |\widetilde F_{2n+1}(x)|&\le \frac{1}{2c_1}
\sum_{k=0}^{n+1-\left[  \frac{n+1}{N}\right]  }
\binom{n+1}{k}\frac{(n+1-k)(2c_{1}c_{2}c_{3})^{n+1-k}\left\vert x-x_{0}\right\vert
^{n+1-k}}{\left(  n+1-k\right)!} \\
    &\le \frac{n+1}{2c_1\bigl(\bigl[\frac {n+1}N\bigr]\bigr)!} \sum_{k=0}^{n+1-\left[  \frac{n+1}{N}\right]  }
\binom{n+1}{k}(2c_{1}c_{2}c_{3})^{n+1-k}\left\vert x-x_{0}\right\vert
^{n+1-k}\le \frac{n+1}{2c_1} C(n+1,x).
\end{align*}
Other inequalities can be obtained similarly.
\end{proof}

\begin{theorem}[Modified SPPS representations for Sturm-Liouville pencils]
\label{Th Modified SPPS copy(1)} Let $p$ and $q$ be such that there exist two
linearly independent solutions $f$ and $g$ of equation \eqref{SL0} such that
$\left\{  f,\,g,\,pf^{\prime},\,pg^{\prime}\right\}  \subset C^{1}[a,b]$ and
$f(x_{0})=g(x_{0})=1$ where $x_{0}$ is any point of $[a,b]$ such that
$p(x_{0})\neq0$. Let the operators $R_{k}$ in \eqref{pencil} be such that
$\left\{  R_{k}[f],\,R_{k}[g]\right\}  \subset C[a,b]$, $k=\overline{1,N}$.
Then the general solution of \eqref{pencil} on $(a,b)$ has the form
\eqref{genmain} where
\begin{equation}
u_{1}=\sum_{n=0}^{\infty}
\lambda^{n}\widetilde{F}_{2n}\qquad\text{and}\qquad u_{2}=
\sum_{n=0}^{\infty}
\lambda^{n}F_{2n+1}. \label{u1u2Pencil}%
\end{equation}
The derivatives of $u_{1}$ and $u_{2}$ have the form%
\begin{equation}
pu_{1}^{\prime}=pf^{\prime}+\sum_{n=1}^{\infty}
\lambda^{n}\left(  pf^{\prime}G_{2n}-\rho\bigl(  pf^{\prime}\widetilde{G}%
_{2n-1}-pg^{\prime}\widetilde{F}_{2n-1}\bigr)  \right)  \label{u1PrimePencil}%
\end{equation}
and
\begin{equation}
pu_{2}^{\prime}=\rho\sum_{n=0}^{\infty}
\lambda^{n}\left(  pg^{\prime}F_{2n}-pf^{\prime}G_{2n}\right)  .
\label{u2PrimePencil}%
\end{equation}
All series in \eqref{u1u2Pencil}--\eqref{u2PrimePencil} converge uniformly on
$[a,b]$ (see also Remark \ref{Rmk P in Der}). The solutions $u_{1}$ and $u_{2}$ satisfy the initial conditions
\begin{equation}
u_{1}(x_{0})=1,\quad u_{1}^{\prime}(x_{0})=f^{\prime}(x_{0}),\quad u_{2}%
(x_{0})=0,\quad u_{2}^{\prime}(x_{0})=\frac{1}{p(x_{0})}.
\label{initialPencil}%
\end{equation}
\end{theorem}

\begin{proof}
Corollary \ref{Corr Estimates Pencil} guarantees the uniform convergence of all the involved series. For example, we have
\begin{align*}
    |u_2|&\le \sum_{n=0}^\infty |\lambda|^n|F_{2n+1}|\le 2c_1c_3\sum_{n=0}^\infty\frac{\bigl(1+2c_1c_2c_3|x-x_0|\bigr)^n|\lambda|^n}{\bigl(\bigl[\frac nN\bigr]\bigr)!}\\
    &= 2c_1c_3\biggl(\sum_{n=0}^{N-1}M^n\biggr)\sum_{m=0}^\infty
    \frac{M^{mN}}{m!}=2c_1c_3\biggl(\sum_{n=0}^{N-1}M^n\biggr)\exp(M^N),
\end{align*}
where $M=\bigl(1+2c_1c_2c_3|x-x_0|\bigr)|\lambda|$.

Due to Lemma \ref{Lemma Derivatives Pencil} we obtain that $u_{1}$ and $u_{2}$
are indeed solutions of (\ref{pencil}) as well as the equalities
(\ref{u1PrimePencil}) and (\ref{u2PrimePencil}). Indeed, let us consider
application of the operator $L$ to $u_{1}$,%
\begin{equation*}
L\left[
\sum_{n=0}^{\infty}
\lambda^{n}\widetilde{F}_{2n}\right]     =\sum_{n=0}^{\infty}
\lambda^{n}\sum_{k=1}^{N}R_{k}\left[  \widetilde{F}_{2n-2k}\right]   =\sum_{k=1}^{N}\lambda^{k}R_{k}\left[
\sum_{n=0}^{\infty}
\lambda^{n-k}\widetilde{F}_{2n-2k}\right]  .
\end{equation*}
Taking into account that the formal powers with negative subindices equal zero
we obtain that $u_{1}$ satisfies (\ref{pencil}). For $u_{2}$ the proof is analogous.

The equalities (\ref{initial}) follow from the fact that all formal powers
$F_{n}$, $G_{n}$, $\widetilde{F}_{n}$ and $\widetilde{G}_{n}$ vanish at
$x=x_{0}$ for any $n\in\mathbb{N}$. Finally, from (\ref{initial}) it follows
that $u_{1}$ and $u_{2}$ are linearly independent.
\end{proof}

\subsection{Spectral shift for pencils}
Let $\lambda_0$ be a fixed complex number and $\lambda=\lambda_0+\Lambda$. The right hand side of equation \eqref{pencil} can be written in the form
\begin{align*}
\sum_{k=1}^{N}\lambda^{k}R_{k}[u]  &  =\sum
_{k=1}^{N}R_{k}[u]\sum_{\ell=0}^{k}\binom{k}{\ell}\lambda_{0}^{\ell}
\Lambda^{k-\ell}\\
&  =\sum_{k=1}^{N}\lambda_{0}^{k} R_{k}[u]+\sum_{k=1}^{N}\Lambda^{k}\sum
_{\ell=0}^{N-k}\binom{k+\ell}{\ell}\lambda_{0}^{\ell}R_{k+\ell}[u],
\end{align*}
therefore equation \eqref{pencil} can be transformed into equation
\begin{equation}\label{pencilshifted}
    L_0u = \sum_{k=1}^{N}\Lambda^{k}\sum
_{\ell=0}^{N-k}\binom{k+\ell}{\ell}\lambda_{0}^{\ell}R_{k+\ell}[u],
\end{equation}
where
\begin{equation*}
    L_0u = (pu')'+qu - \sum_{k=1}^{N}\lambda_{0}^{k} R_{k}[u] = (pu')'+u\biggl(q - \sum_{k=1}^{N}\lambda_{0}^{k} r_{k}\biggr)-u'\sum_{k=1}^{N}\lambda_{0}^{k} s_{k}.
\end{equation*}
Equation \eqref{pencilshifted} is of the form \eqref{pencil} only for some special cases, say all the coefficients $s_k$ are identically zeros or the coefficients $s_k$ are linearly dependent and such that for some special values of $\lambda_0$ the expression $\sum_{k=1}^{N}\lambda_{0}^{k} s_{k}$ equals zero. In other situations equation \eqref{pencilshifted} has nonzero coefficient near $u'$. To overcome this difficulty we multiply all terms of equation \eqref{pencilshifted} by
\[
P(x):=\exp\biggl(-\int_{x_0}^x \frac 1{p(s)}\sum_{k=1}^N \lambda_0^k s_k(s)\,ds\biggr),
\]
and transform it into the equation
\begin{equation}\label{pencilshifted2}
    (\widetilde pu')' +\widetilde q u = \sum_{k=1}^{N}\Lambda^{k}\widetilde R_{k}[u],
\end{equation}
where
\begin{equation}\label{pencilshiftedLcoeffs}
    \widetilde p=p\cdot P,\qquad \widetilde q = P\biggl(q-\sum_{k=1}^N\lambda_0^k r_k\biggr)
\end{equation}
and
\begin{equation}\label{pencilshiftedRcoeffs}
    \widetilde  R_k[u]=\widetilde r_k u+\widetilde s_ku'\quad\text{with}\quad
    \widetilde r_k=P\cdot\sum_{\ell=0}^{N-k}\binom{k+\ell}{\ell}\lambda_0^\ell r_{k+\ell},\quad \widetilde s_k=P\cdot\sum_{\ell=0}^{N-k}\binom{k+\ell}{\ell}\lambda_0^\ell s_{k+\ell}.
\end{equation}

Note that a particular solution of \eqref{pencilshifted2} corresponding to $\Lambda=0$ is the particular solution of \eqref{pencil} corresponding to $\lambda=\lambda_0$. Hence applying Theorem \ref{Th Modified SPPS copy(1)} to equation \eqref{pencilshifted2} and taking into account that $\widetilde p(x_0) = p(x_0)$ we obtain the following corollary.
\begin{corollary}[Spectral shift for the modified SPPS representation]
Let equation \eqref{pencil} admit for $\lambda=\lambda_0$ two linearly independent solutions $f$ and $g$ such that $\left\{  f,\,g,\,pf^{\prime},\,pg^{\prime}\right\}  \subset C^{1}[a,b]$ and
$f(x_{0})=g(x_{0})=1$ where $x_{0}$ is any point of $[a,b]$ such that
$p(x_{0})\neq0$. Let $\frac 1p\sum_{k=1}^N\lambda_0^k s_k\in C[a,b]$ and
$\left\{  R_{k}[f],\,R_{k}[g]\right\}  \subset C[a,b]$, $k=\overline{1,N}$.
Then the general solution of \eqref{pencil} on $(a,b)$ has the form
\eqref{genmain} where
\begin{equation}
u_{1}=\sum_{n=0}^{\infty}
(\lambda-\lambda_0)^{n}\widetilde{F}_{2n}\qquad\text{and}\qquad u_{2}=
\sum_{n=0}^{\infty}
(\lambda-\lambda_0)^{n}F_{2n+1}, \label{u1u2PencilShifted}
\end{equation}
and the functions $\{F_n\}$ and $\{\widetilde F_n\}$ are obtained by applying formulas from Definition \ref{Def Powers Pencil} to the functions $f$, $g$ and $\widetilde R_k[f]$, $\widetilde R_k[g]$.

The derivatives of $u_{1}$ and $u_{2}$ have the form
\begin{equation}
pu_{1}^{\prime}=pf^{\prime}+\sum_{n=1}^{\infty}
(\lambda-\lambda_0)^{n}\left(  pf^{\prime}G_{2n}-\rho\bigl(  pf^{\prime}\widetilde{G}%
_{2n-1}-pg^{\prime}\widetilde{F}_{2n-1}\bigr)  \right)  \label{u1PrimePencilShifted}%
\end{equation}
and
\begin{equation}
pu_{2}^{\prime}=\rho\sum_{n=0}^{\infty}
(\lambda-\lambda_0)^{n}\left(  pg^{\prime}F_{2n}-pf^{\prime}G_{2n}\right)  .
\label{u2PrimePencilShifted}%
\end{equation}
All series in \eqref{u1u2PencilShifted}--\eqref{u2PrimePencilShifted} converge uniformly on
$[a,b]$. The solutions $u_{1}$ and $u_{2}$ satisfy the same initial conditions \eqref{initialPencil}.
\end{corollary}

\section{Numerical solution of spectral problems}\label{Section4}
\subsection{The general scheme}
The general scheme of using the modified SPPS representation for the solution of spectral problems for equation \eqref{SL} and more general \eqref{pencil} is similar to that for the original SPPS representation, see \cite{KrPorter2010}, \cite{KrT2013}.

Consider boundary conditions
\begin{gather}\label{BC1}
    \alpha_a u(a) + \beta_a p(a)u'(a)=0\\
    \alpha_b u(b) + \beta_b p(b) u'(b)=0,\label{BC2}
\end{gather}
where $\alpha_a$, $\beta_a$, $\alpha_b$ and $\beta_b$ are complex numbers such that $|\alpha_a|+|\beta_a|\ne 0$ and $|\alpha_b|+|\beta_b|\ne 0$. Suppose that the function $p$ is continuous at one of the endpoints and is different from zero at that endpoint. We may assume that $a$ is such endpoint. Let $f$ and $g$ be two linearly independent solutions of \eqref{SL0} satisfying $f(a)=g(a)=1$, and denote $h:=f'(a)$. Consider the systems of functions $\{  F_{n}\}  $, $\{  \widetilde{F}_{n}\}$, $\{  G_{n}\}$, $\{\widetilde{G}_{n}\}$ constructed from the solutions $f$ and $g$ by Definition \ref{Def Systems Fn} or by Definition \ref{Def Powers Pencil} using the point $x_0=a$.
Then due to the initial conditions \eqref{initial} or \eqref{initialPencil} the solution $u(x;\lambda)$ defined by
\begin{equation*}
    u(x;\lambda)=\beta_a u_1(x;\lambda) - (\alpha_a+\beta_a h) u_2(x;\lambda),
\end{equation*}
where the functions $u_1$ and $u_2$ are given by \eqref{u1u2} or \eqref{u1u2Pencil}, satisfies the first boundary condition \eqref{BC1}. Hence the second boundary condition \eqref{BC2} gives us the characteristic function \begin{equation}\label{CharFun}
    \Phi(\lambda):=\alpha_b u(b;\lambda) + \beta_b p(b) u'(b;\lambda).
\end{equation}
The set of zeros of the function $\Phi$ coincides with the set of eigenvalues of the spectral problem \eqref{BC1}, \eqref{BC2} for the equation \eqref{pencil}. Truncating the series in \eqref{CharFun} we obtain a polynomial approximating the characteristic function. The roots of this polynomial closest to zero give us approximations of the eigenvalues. The Rouche theorem guarantees that these roots are indeed the approximations to the eigenvalues and are not spurious roots appearing as a result of the truncation of the series.

In the case when the function $p$ is not continuous or equals zero at the endpoints, we cannot calculate the formal powers starting from one of the endpoints and cannot take advantage of the initial conditions \eqref{initial} or \eqref{initialPencil}. Instead we consider the general solution $u=c_1u_1 + c_2u_2$ constructed using some point $x_0\in (a,b)$. Then a point $\lambda$ is an eigenvalue of the problem if and only if the determinant of the following system
\begin{equation}\label{NastyCharEq}
    \det \left(
           \begin{array}{cc}
             \alpha_a u_1(a;\lambda)+\beta_ap(a) u_1'(a;\lambda) & \alpha_a u_2(a;\lambda)+\beta_ap(a) u_2'(a;\lambda) \\
             \alpha_b u_1(b;\lambda)+\beta_bp(b) u_1'(b;\lambda) & \alpha_b u_2(b;\lambda)+\beta_bp(b) u_2'(b;\lambda)
           \end{array}
         \right)=0,
\end{equation}
is equal to zero, see, e.g., \cite[\S 1.3]{Marchenko},
and we can proceed as before: taking the partial sums of the involved series, obtaining a polynomial approximating the characteristic equation and choosing the roots closest to zero.

\subsection{Numerical examples for Sturm-Liouville problems}
In the paper \cite{KrPorter2010} the authors illustrated the numerical performance of the SPPS method for solving Sturm-Liouville spectral problems. Since the difference between the original SPPS representation and the modified SPPS representation consists only in the way of calculating coefficients, the performance of the modified SPPS method is similar to that of the SPPS method when all the involved recursive integrals can be calculated equally precise. Usually it is the case when a particular solution $f$ and functions $1/p$, $r$ do not grow rapidly and are sufficiently separated from zero. In the opposite case one may expect a better performance of the modified SPPS method. One of the examples with a rapidly growing particular solution $f$, the Coffey-Evans equation, is considered in \cite{KrT2013} where we observe that a combination of the Clenshaw-Curtis integration formula with the formulas \eqref{D2}--\eqref{D8} allows us to compute twice as many formal powers in comparison with the formulas \eqref{Xgen2}, \eqref{Xgen3}.

In this subsection we consider several ``nasty'' examples (according to \cite[Appendix B]{Pryce}) involving unbounded however absolutely integrable functions $1/p$, $r$, $q$. Even though some of the problems do not satisfy the conditions of Theorem \ref{Th Modified SPPS}, the modified SPPS method demonstrates an excellent accuracy, meanwhile the performance of the SPPS method is considerably worse for the problems with unbounded functions $1/p$ or $r$. Moreover, the numerical implementation of the SPPS method is several times slower for these problems due to the necessity to use complex-valued functions in order to obtain non-vanishing particular solutions.

\begin{example}\label{Ex10Pryce}
Consider the following problem (Problem 10 from \cite{Pryce})
\begin{equation*}\label{EqEx10Pryce}
\begin{cases}
-\left(\sqrt{1-x^2}u'\right)'=\lambda u,\\
\sqrt{1-x^2}u'(x)\big|_{x=-1}=0,\quad u(1)=0,
\end{cases}
\end{equation*}
a problem with  a ``nasty'' $p=\sqrt{1-x^2}$ and ``good'' $q$ and $r$.

Since the function $p$ equals zero at both endpoints, we used the determinant approach described in the previous subsection.

The functions $f(x)=1$ and $g(x)=1+\arcsin(x)$ were chosen as two particular solutions of equation \eqref{SL0} satisfying the conditions of Theorem \ref{Th Modified SPPS}.

We obtained approximate eigenvalues of the problem applying the spectral shift technique, on each step finding one new approximate eigenvalue as the root of the polynomial approximating the characteristic equation closest to the current spectral shift center and using this value as the spectral shift for the next step. On each step we computed $N=100$ formal powers using machine precision arithmetics in MATLAB with $x_0=0$ and $M=2\cdot 10^5-1$ points for the Newton-Cottes 6 points integration scheme. We also tested the ``old'' SPPS method on this problem. In order to deal with the zeros of the function $p$ at the endpoints we approximated it by a function having small, however non-zero values at the endpoints. The results from the SPPS representation were obtained using the same parameters and the strategy for the spectral shift, with the only difference that we have taken a complex-valued combination $u_1+iu_2$ on each step for a particular solution to be non-vanishing. The obtained results are presented in Table \ref{Ex10PryceTable1} together with the values from \cite{Pryce} and the results produced by SLEIGN2 package \cite{BEZ2}. Another well-known package, MATSLISE \cite{LVV}, can not solve this problem at all. Unfortunately the exact characteristic equation for this problem is unknown. Note that the results of the modified SPPS method are in a good agreement with those presented in \cite{Pryce}, meanwhile the results produced by SLEIGN2 differ in 3rd--5th decimal place, the results of the SPPS method are even worse.

\begin{table}
\centering
\begin{tabular}
{ccccc}\hline
$n$ & $\lambda_n$ (\cite{Pryce}) & $\lambda_{n}$ (our method) & $\lambda_{n}$ (``old'' SPPS method) & $\lambda_n$ (SLEIGN2) \\\hline
0 & 0.3856819 & 0.385681872027002 & 0.3863 & 0.385684539\\
1 & & 3.80741155419017  & 3.8114 & 3.807427952\\
2 & &10.6772827352614  & 10.6867 & 10.677320922\\
3 & &20.9871308475868  & 21.0036 & 20.987197576\\
5 & &51.9221036193997  & 51.9570 &  51.922245020\\
10 & &189.421910262487  &  189.5241 & 189.422324959\\
15 & &412.863500805267  & 413.0592 & 412.864294034\\
20 & &722.245619500433  & 722.5567 & 722.246883258\\
24 & 1031.628 & 1031.62824937392 &  1032.047 & 1031.629950116\\
\hline
\end{tabular}
\caption{The eigenvalues of the Problem 10 from \cite{Pryce} (Example \ref{Ex10Pryce}).}
\label{Ex10PryceTable1}%
\end{table}
\end{example}

\begin{example}\label{Ex09Pryce}
Consider the following problem (Problem 9 from \cite{Pryce}). The interval is $[-1,1]$, ``nice'' $p=1/\sqrt{1-x^2}$ and $q=0$, ``nasty'' $r=1/\sqrt{1-x^2}$ with the Dirichlet boundary conditions $u(-1)=u(1)=0$.

We tested the performance of the Darboux-associated equations approach proposed in Subsection \ref{subsect Darboux} and Remark \ref{Rmk Darboux} on this problem. Even using the spectral shift technique, the results for the higher eigenvalues were mediocre, see Table \ref{Ex09PryceTable1}. Such behavior of the method can be explained by the additional steps related with the Darboux associated equations, namely construction of the potentials $q_{1/f}$ and of a second particular solution of these associated equations. Obtained potentials $q_{1/f}$ possessed large peaks inside the interval leading to large errors in the calculated formal powers.

Additionally we applied the direct approach to check whether our method can be applied in the situations not covered by Theorem \ref{Th Modified SPPS}. For that we chose $f(x)=1$ and $g(x)=1+\bigl(x\sqrt{1-x^2}+\arcsin x\bigr)/2$ as particular solutions of \eqref{SL0} satisfying the conditions of Theorem \ref{Th Modified SPPS}, changed values of $r$ at the endpoints to be equal to some rather large values and proceeded exactly as described in Example \ref{Ex10Pryce}. The obtained results are presented in Table \ref{Ex09PryceTable1} and are in an excellent agreement with those reported in \cite{Pryce}. Some of the eigenvalues computed by SLEIGN2 package differ from our results in 3-5th decimal place. Also we tested the performance of the SPPS method. Produced eigenvalues are closer than in the previous example to the obtained by the modified SPPS method and agree up to 4-6 decimal places.
\begin{table}
\centering
\begin{tabular}
{ccccc}\hline
$n$ & $\lambda_n$ (\cite{Pryce}) & $\lambda_{n}$ (our method) & $\lambda_{n}$ (our method, based on&  $\lambda_n$ (SLEIGN2) \\
& & & Darboux-associated eqns.) &\\
\hline
0 & 3.559279966 & 3.55927997532677 &  3.559280003 &   3.559279975351\\
1 & & 12.1562946865237 & 12.15629481 &   12.15637\\
2 & & 25.7034532288478 & 25.70345354 &   25.70345322896\\
3 & & 44.1919717455476 & 44.19197235 &   44.19206\\
5 & & 95.9831209203069 & 95.98312252 &   95.98332\\
9 & 258.8005854& 258.800585373152 & 258.8005909 &  258.7976\\
14 & & 573.369367026965 & 573.36944  &  573.3693670289 \\
19 & & 1011.31532988447 & 1011.19    & 1011.3153298853\\
24 & 1572.635284 & 1572.63528434735 & --         & 1572.6352843481\\
\hline
\end{tabular}
\caption{The eigenvalues of the Problem 9 from \cite{Pryce} (Example \ref{Ex09Pryce}).}
\label{Ex09PryceTable1}%
\end{table}
\end{example}

\begin{example}\label{Ex11Pryce}
Consider the following problem (Problem 11 from \cite{Pryce})
\begin{equation*}\label{EqEx11Pryce}
\begin{cases}
-u''+u\ln x=\lambda u,\\
u(0)=u(4)=0.
\end{cases}
\end{equation*}
Again, this problem is not covered by Theorem \ref{Th Modified SPPS}. Nevertheless we checked the performance of our method on this problem. Two particular solutions of equation \eqref{SL0} were computed using the SPPS representation. After that we proceeded exactly as in Examples \ref{Ex10Pryce} and \ref{Ex09Pryce} using the point $x_0=2$ to calculate the formal powers. We also checked the performance of the SPPS method. Obtained results together with the results from \cite{Pryce} and the results produced by SLEIGN2 package are presented in Table \ref{Ex11PryceTable1}.
\begin{table}
\centering
\begin{tabular}
{ccccc}\hline
$n$ & $\lambda_n$ (\cite{Pryce}) & $\lambda_{n}$ (our method) & $\lambda_{n}$ (old SPPS method) & $\lambda_n$ (SLEIGN2) \\
\hline
0 & 1.1248168097 & 1.12481680968989 &  1.1248168096898 & 1.12481680982\\
1 & & 2.99094198359879  & 2.99094198359867 &  2.990941998\\
2 & & 6.03307162455419  &   6.03307162455413 &  6.03307134\\
4 & & 15.8644572215756  &  15.8644572215752 & 15.86445693\\
9 & & 62.0987975024207  &   62.0987975024165 & 62.0987975072\\
24 & 385.92821596 & 385.928215961012 & 385.928215961016 & 385.928215990\\
\hline
\end{tabular}
\caption{The eigenvalues of the Problem 11 from \cite{Pryce} (Example \ref{Ex11Pryce}).}
\label{Ex11PryceTable1}%
\end{table}
\end{example}

\subsection{High-precision evaluation of eigenvalues}
In this subsection we show that the modified SPPS method can be successfully applied to the calculation of eigenvalues of Sturm-Liouville spectral problems with a high accuracy. However, in contrast to the method proposed in \cite{KrT2013}, the accuracy of the eigenvalues rapidly deteriorates with the eigenvalue index. The situation can be improved to some extent applying the spectral shift technique allowing one to obtain hundreds of highly accurate approximate eigenvalues.

\begin{example}\label{ExHighAccuracy}
Consider the following spectral problem (the second Paine problem, \cite{Paine, Pryce})
\begin{equation*}
\begin{cases}
-u''+\frac{1}{(x+0.1)^2} u=\lambda u, & 0\le x\le \pi,\\
u(0,\lambda)=0,\quad  u(\pi,\lambda)=0. &
\end{cases}
\end{equation*}
This problem was treated in \cite{KrT2013} and appears to be rather tough requiring a large number of formal powers to be used in order to compute highly accurate eigenvalues. In \cite{KrT2013} we were able to achieve the accuracy of order $10^{-43}\div10^{-42}$ almost independent of the eigenvalue index for several thousands of eigenvalues. Further increase of accuracy required significant increase of all the parameters involved (number of the formal powers, precision and the number of points used for the integration). In this example we show that the modified SPPS method allows us to improve the accuracy to the order of $10^{-150}$ using the similar set of parameters however only for the first 187 eigenvalues.

First we verified the precision of the coefficients of the polynomial approximating the exact characteristic function. These coefficients are nothing more than the values of the formal powers at the right endpoint divided by the corresponding factorials. We compared the different methods of indefinite numerical integration used for evaluating the formal powers. Up to now we used three different methods of indefinite numerical integration, see \cite{CKT2013}, \cite{KKTT} and \cite{KrT2013}. The first is the modification of the Newton-Cottes 7th order six point rule, the second is the integration of a spline approximating a formal power and the third is the Clenshaw-Curtis integration based on the approximation of a function by the Tchebyshev polynomials. The computation time required by the second mentioned method highly exceeds the computation time required by the first method providing only a slight improvement of the accuracy. For that reason in the present work we consider only the first and the third integration methods. All the computations were performed in Wolfram Mathematica 8.

For each of the methods a parameter $M$ corresponds to the number of smaller subdivision intervals on the segment $[0,\pi]$ used for numerical integration, i.e., the integrand function was represented by its values in $M+1$ points. For the Clenshaw-Curtis integration we used for $M$ values $512$, $1024$, $2048$ and $3072$. For each of the values of $M$ we computed two particular solutions using the SPPS representation and verified their precision against the exact particular solution $u_0(x)=(1+10x)^{(1+\sqrt{5}) /2}$. The maximum absolute errors were $3.9\cdot 10^{-85}$, $7.5\cdot 10^{-165}$, $1.7\cdot 10^{-323}$ and $8.2\cdot 10^{-482}$ respectively. Therefore we used $100$, $200$, $400$ and $600$ digit arithmetic respectively for the calculation of the formal powers.

For the Newton-Cottes integration scheme we used $M=10^4$, $5\cdot 10^4$ and $25\cdot 10^4$ and performed computations in machine-precision and $64$-digit arithmetics, in both cases using exact particular solutions.

We compared the computed coefficients (values of the formal powers at the right endpoint divided by the corresponding factorials) against the same values produced by means of the Clenshaw-Curtis integration formula with $M=4096$. The relative errors of the formal powers are presented on Figure \ref{ExHPFigCoeffs}. Note the different behavior of the errors. For the Clenshaw-Curtis integration the errors start from much lower values coinciding with the errors of the particular solutions, however rapidly increasing with the increase of the formal power number. For the Newton-Cottes integration the errors in machine-precision are almost constant and are slowly growing in the high precision arithmetic.

\begin{figure}[h]
\centering
\begin{tabular}{cc}
\includegraphics[
height=1.426in,
width=3in,
bb=0 0 305 145
]
{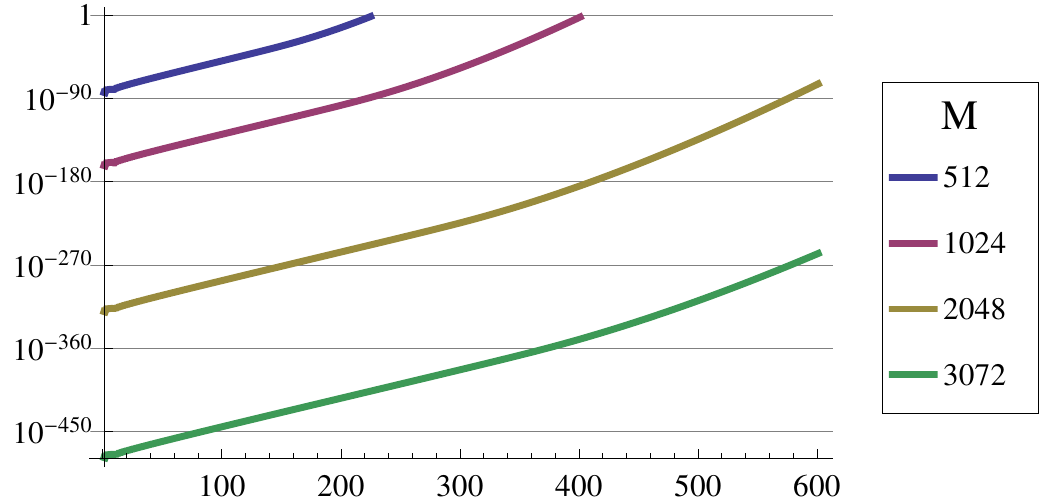}
&
\includegraphics[
height=1.425in,
width=3in,
bb=0 0 300 144
]
{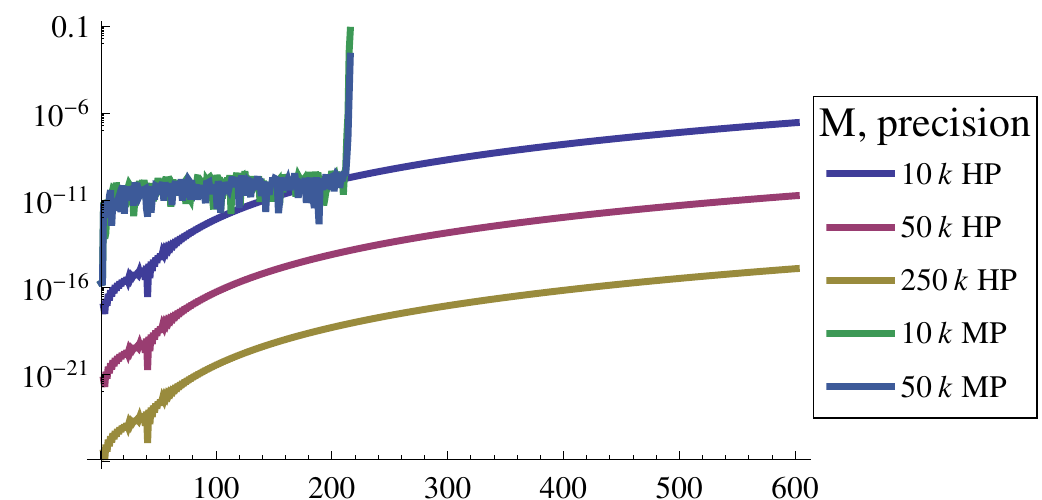}
\end{tabular}
\caption{Relative errors of the first 600 formal powers in Example \ref{ExHighAccuracy} obtained using Clenshaw-Curtis integration (on the left graph) and using the Newton-Cortes integration (on the right graph). $M$ corresponds to the number of points used for representing the integrand, HP means 64 digit precision and MP means machine precision.}\label{ExHPFigCoeffs}
\end{figure}

Using the obtained coefficients we calculated the roots of the polynomial approximating eigenvalues and compared them to the exact ones (see \cite[Example 26]{KrT2013} for the expression of the characteristic equation). Since the problem possesses only real eigenvalues, all roots of the polynomial having large imaginary part were discarded as spurious roots. On Figure \ref{ExHPFigRoots} we present the graphs of the absolute errors of the approximate eigenvalues obtained from the truncation of the modified SPPS representation using $N=100$, $200$, $400$ and $600$ formal powers and without application of the spectral shift.


\begin{figure}[h]
\centering
\begin{tabular}{cc}
\includegraphics[
height=2in,
width=3in,
bb=0 0 216 144
]
{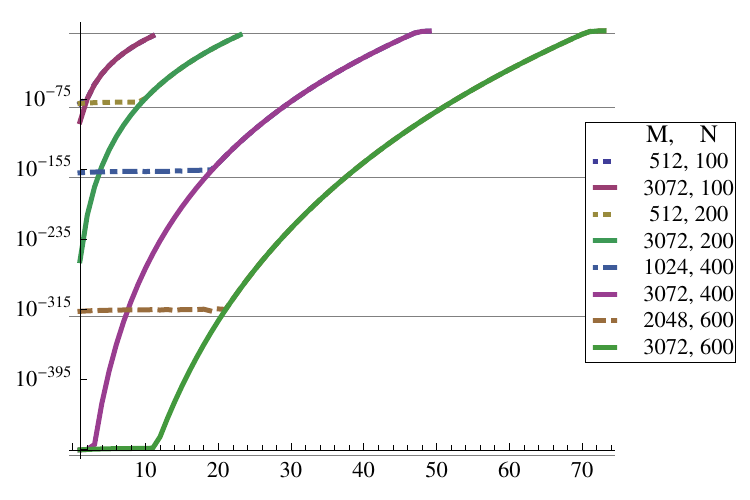}
&
\includegraphics[
height=2in,
width=3in,
bb=0 0 216 144
]
{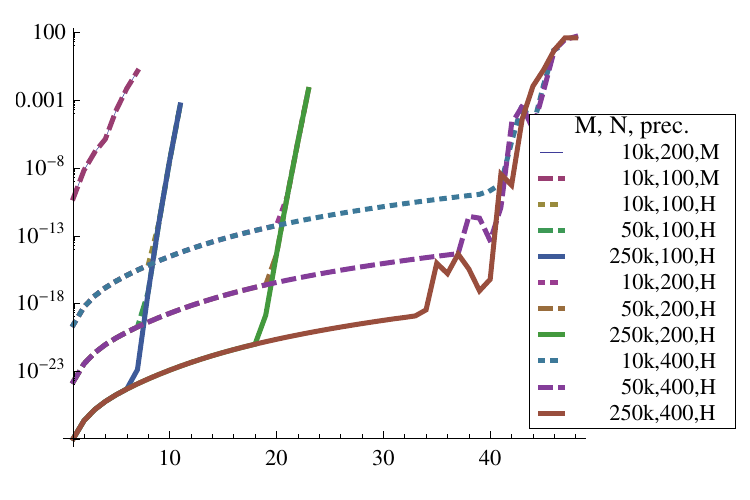}
\end{tabular}
\caption{Absolute errors of the approximate eigenvalues in Example \ref{ExHighAccuracy} obtained using different number of formal powers for approximating the exact characteristic equation (parameter $N$) and using Clenshaw-Curtis integration (on the left graph) and using the Newton-Cortes integration (on the right graph). $M$ corresponds to the number of points used for representing the integrand, H means 64 digit precision and M means machine precision. The horizontal lines on the left graph show the errors of the particular solutions used for the calculation of the formal powers.}\label{ExHPFigRoots}
\end{figure}

Several observations can be made regarding the presented graphs. First, the number of eigenvalues which can be approximately calculated from the truncated SPPS representation depends on the number of used formal powers and almost does not depend on the accuracy of the formal powers. Second, the accuracy of the formal powers has a great influence on the accuracy of the first eigenvalues. The errors of the first approximate eigenvalues are close to the errors achieved while calculating the particular solutions and the first several formal powers, meanwhile the errors of the larger eigenvalues remain roughly constant for different computation precisions used.

\begin{figure}[h]
\centering
\includegraphics[
height=2.8in,
width=4.8in,
bb=0 0 412 240
]
{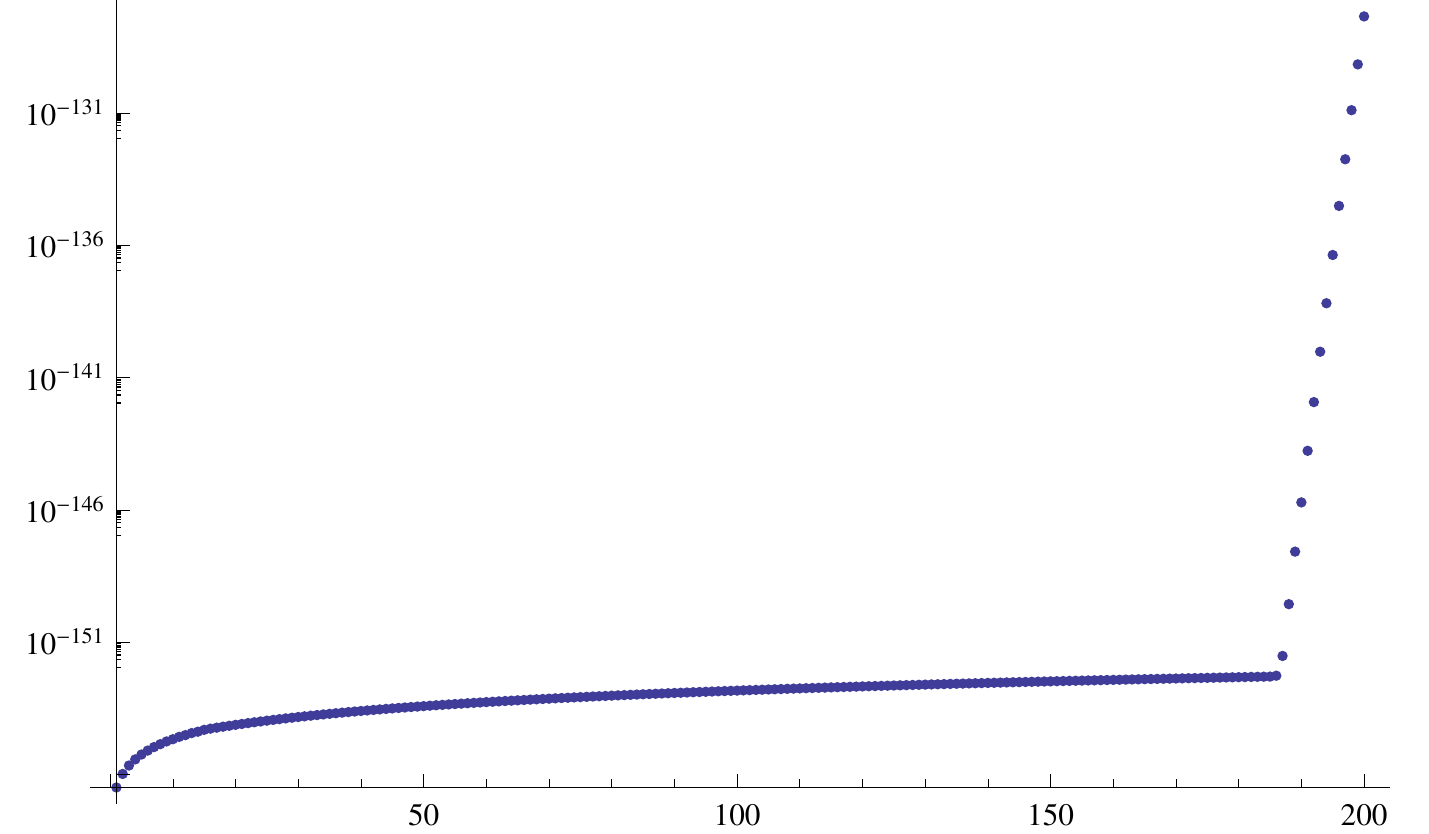}
\caption{Absolute errors of the first 200 approximate eigenvalues in Example \ref{ExHighAccuracy} obtained by the modified SPPS method applying the spectral shift technique.}\label{ExHPFigEigenvalues}
\end{figure}

Finally we computed the approximate eigenvalues applying the spectral shift technique. We performed spectral shifts using values $\lambda_0=250n$, $n=1,\ldots,200$ and on each step calculating $N=400$ formal powers with the help of the Clenshaw-Curtis integration with $M=1024$ and 200-digit arithmetic. The absolute errors of the first 200 found eigenvalues are presented on Figure \ref{ExHPFigEigenvalues}. As one can see, the errors are slowly growing remaining smaller than $10^{-150}$ up to the eigenvalue number 186, for the higher indices the accuracy rapidly deteriorates.
\end{example}

\subsection{Spectral problems for pencils}
In this subsection we consider several examples in which the right-hand side of equation \eqref{pencil} includes a derivative of the unknown function at the spectral parameter or depends polynomially on the spectral parameter.

The first two considered problems are from \cite{AT2007}, \cite{AT2013} and belong to so-called second-order linear pencils.

\begin{example}\label{ExLinearPencil3}
Consider the following problem \cite[Example 3.3]{AT2013}.
\begin{equation}
\label{EqLinearPencil3}
\begin{cases}
-y''+x^2 y=\lambda(2iy'+y), & 0\le x\le 1,\\
y'(0)+i\lambda y(0)=0, & y'(1)+i\lambda y(1)=0.
\end{cases}
\end{equation}
The problem is self-adjoint and possesses a discrete real spectrum. With the help of Mathematica software we found the characteristic equation of the problem is given by the expression
\[
\left(\lambda
   ^2+\lambda -1\right) \, _1F_1\left(\frac{1}{4}
   (5-\lambda  (\lambda
   +1));\frac{3}{2};1\right)+\,
   _1F_1\left(\frac{1}{4} (1-\lambda  (\lambda
   +1));\frac{1}{2};1\right)=0,
\]
where $_1F_1$ is the Kummer confluent hypergeometric function.

We computed two particular solutions of \eqref{SL0} using the SPPS representation with $N=100$ formal powers and $M=10001$ points for the evaluation of the involved integrals by the Newton-Cottes 6 point formula, afterwards we used these particular solutions to compute $N=100$ formal powers and to find the roots of the polynomial approximating the exact characteristic equation, spectral shift technique was used to obtain the higher index eigenvalues. The obtained eigenvalues together with the exact ones and the results from \cite{AT2007} and \cite{AT2013} are presented in Table \ref{TabLinearPencil3}. Note that our results are significantly better than the results from \cite{AT2007} and are comparable with the ones from \cite{AT2013}. However it should be mentioned that the approximations of the characteristic function of the problem \eqref{EqLinearPencil3} from \cite{AT2007} and \cite{AT2013} do not lead to an automatic approximation of the eigenfunctions; require some analytic precomputation as well as the solution of a large number of initial value problems which the authors of \cite{AT2007} and \cite{AT2013} performed by means of Mathematica with a required accuracy. Meanwhile the results delivered by the modified SPPS method were obtained using machine precision, did not require any analytic precomputation and include the eigenfunctions as well.
\begin{table}
\centering
\begin{tabular}
{ccccc}\hline
$n$ & $\lambda_n$ (our method) & $\lambda_{n}$ (exact) & $\lambda_{n}$ (\cite{AT2007}) & $\lambda_n$ (\cite{AT2013}) \\
\hline
-25& -75.90209254554286 & -75.90209254550119 &  & \\
-10& -28.78465916307922 & -28.78465916308716 &  & \\
-5 & -13.08969157402720 & -13.08969157402805 &  & \\
-3 & -6.830508103259227 & -6.830508103259007 &  & \\
-2 & -3.741923372554198 & -3.741923372554521 &  -3.7419233703827506 & -3.7419233725545213\\
-1 & -1.258249036460409 & -1.2582490364604132 & -1.2582490390569894 &  -1.2582490364604124\\
0 & 0.258249036460413 & 0.2582490364604132  &   0.2582490344106217 &  0.25824903646041525\\
1 & 2.741923372554577 & 2.741923372554521  &  2.741923371301097 & 2.7419233725545213\\
2 &  5.830508103259199& 5.830508103259007  &   5.830508103873908 & 5.8305081032590085\\
3 & 8.955988815983204 & 8.955988815983707 &  & \\
5 & 15.22658797653006 & 15.22658797653187 &  & \\
10& 30.92521763113015 & 30.92521763112857 &  & \\
25& 78.04353040058767 & 78.04353040632336 &  & \\
\hline
\end{tabular}
\caption{The eigenvalues of the problem \eqref{EqLinearPencil3} (Example \ref{ExLinearPencil3}).}
\label{TabLinearPencil3}
\end{table}
\end{example}

\begin{example}\label{ExLinearPencil1}
Consider the following boundary value problem \cite[Example 3.1]{AT2013}.
\begin{equation}
\label{EqLinearPencil1}
\begin{cases}
-y''+q(x) y=\lambda(2iy'+y), & 0\le x\le 1,\\
y(0)=0,\quad y'(1)+i\lambda y(1)=0, &
\end{cases}
\end{equation}
where
\[
q(x)=\begin{cases}
1, & 0\le x\le 1/2,\\
0, & 1/2<x\le 1.
\end{cases}
\]
This problem is not covered by Theorem \ref{Th Modified SPPS copy(1)}, however it can be solved by the modified SPPS representation according to Remark \ref{RmkPiecewiseContinuousCoeffs}. There seems to be some error in \cite{AT2007}, \cite{AT2013} because the reported results are not the eigenvalues of the problem \eqref{EqLinearPencil1}. With the help of Wolfram Mathematica we found that the characteristic equation of the problem \eqref{EqLinearPencil1} is given by the expression
\begin{equation}\label{EqLinearPencil1ChEq}
\sqrt{\lambda ^2+\lambda}
   \tanh \left(\frac{1}{2} \sqrt{-\lambda
   (\lambda +1)}\right) \tanh
   \left(\frac{1}{2} \sqrt{1-\lambda
   (\lambda +1)}\right)+\sqrt{\lambda ^2+\lambda -1}=0.
\end{equation}

We applied the modified SPPS method to this problem using the spectral shift technique computing both the particular solutions and the first $100$ formal powers using $M=10001$ for all involved integrals and performing integrations separately on each segment of continuity of the potential $q$. The calculated eigenvalues together with the exact ones obtained from \eqref{EqLinearPencil1ChEq} and with the resulted absolute errors are presented in Table \ref{TabLinearPencil1}.
\begin{table}
\centering
\begin{tabular}
{cccc}\hline
$n$ & $\lambda_n$ (our method) & $\lambda_{n}$ (exact) & Abs. error \\
\hline
-25 & -77.4738498134661 & -77.4738498206540 & $7.8\cdot 10^{-9}$ \\
-10 & -30.3579741391681 & -30.3579741391157 & $6.2\cdot 10^{-11}$ \\
-5 & -14.6624304044055 &  -14.6624304044072 & $1.9\cdot 10^{-12}$ \\
-3 & -8.39761752583675 &  -8.39761752583497 & $3.9\cdot 10^{-12}$ \\
-2 & -5.30260260783015 &  -5.30260260783027 & $2.5\cdot 10^{-12}$ \\
-1 &  -2.20110385479012 &  -2.20110385479002  & $1.1\cdot 10^{-13}$\\
0 & 1.20110385479006 & 1.20110385479002  & $3.7\cdot 10^{-14}$  \\
1 & 4.30260260783056& 4.30260260783027  & $2.8\cdot 10^{-13}$ \\
2 &  7.39761752583498& 7.39761752583497  & $1.4\cdot 10^{-12}$\\
3 &  10.5317097032223& 10.5317097032191  & $3.5\cdot 10^{-12}$\\
5 &  16.8012911248982& 16.8012911248964  & $1.0\cdot 10^{-11}$\\
10 & 32.4978603143171& 32.4978603143055  & $1.2\cdot 10^{-11}$\\
25 & 76.4738498191705& 76.4738498206540  & $7.0\cdot 10^{-9}$\\
\hline
\end{tabular}
\caption{The eigenvalues of the problem \eqref{EqLinearPencil1} (Example \ref{ExLinearPencil1}).}
\label{TabLinearPencil1}
\end{table}
\end{example}

For the next example we considered the following boundary value problem
\begin{equation*}
    \begin{cases}
    \frac{\partial}{\partial s}\left(A(s)\frac{\partial u}{\partial s}\right)-\frac{\partial^2 u}{\partial t^2}-p(s)\frac{\partial u}{\partial t}=0,\\
    u(0,t)=0,\\
    \left.\frac{\partial u}{\partial s}\right|_{s=l}+\nu\left.\frac{\partial u}{\partial t}\right|_{s=l}+\mu \left.\frac{\partial^2 u}{\partial t^2}\right|_{s=l}=0,
    \end{cases}
\end{equation*}
describing small transverse vibrations of a string of stiffness $A(s)$ with a damping coefficient $p(s)>0$. Here $u(s,t)$ is the transverse displacement and $l>0$ is the length of the string. The left end of the string is fixed and the right end is equipped with a ring of mass $\mu>0$ moving in the direction orthogonal to the equilibrium position of the string. The damping coefficient of the ring is $\nu>0$. Similar problems were considered in various papers where theoretical results on direct and inverse problems were obtained, see, e.g., \cite{Jaulent1976}, \cite{Pivovarchik1999}, \cite{Pivovarchik2007}. Substituting $u(s,t)=v(\lambda,s)e^{i\lambda t}$ we obtain the system for the amplitude function $v(\lambda, s)$.
\begin{equation}\label{ExStringEq}
    \begin{cases}
    \bigl(A(s)v'(\lambda,s)\bigr)' +\lambda^2 v(\lambda,s)-ip(s)\lambda v(\lambda, s)=0,\\
    v(\lambda, 0)=0,\\
    v'(\lambda, l)+i\nu \lambda v(\lambda,l)-\mu\lambda^2 v(\lambda, l)=0.
    \end{cases}
\end{equation}
The equation in \eqref{ExStringEq} is of the type \eqref{pencil}. In the case of a constant $p(s)\equiv p$ the problem can be reduced to a Sturm-Liouville problem by a change of the spectral parameter, however for a non-constant damping $p(s)$ the equation should be solved as a pencil.

\begin{example}\label{ExString}
To be able to compare the approximate eigenvalues produced by the modified SPPS method with the exact ones we have chosen the following parameters: $A(s)\equiv 1$, $p(s)=s$, $\mu=\nu=1$ and $l=1$. For these parameters we were able to find with the help of Mathematica software the exact characteristic equation
\begin{multline}\label{ExStringCharEq}
    \frac{\pi}{\sqrt[3]{i \lambda
   }}  \left(\operatorname{Bi}\bigl((i \lambda
   )^{4/3}\bigr) \left(\lambda  (\lambda -i)
   \operatorname{Ai}\bigl((i \lambda +1) \sqrt[3]{i
   \lambda }\bigr)-\sqrt[3]{i \lambda }
   \operatorname{Ai}'\bigl((i \lambda +1) \sqrt[3]{i
   \lambda }\bigr)\right)+\right.\\
   \left.\operatorname{Ai}\bigl((i
   \lambda )^{4/3}\bigr) \left(\sqrt[3]{i
   \lambda } \operatorname{Bi}'\bigl((i \lambda +1)
   \sqrt[3]{i \lambda }\bigr)-\lambda
   (\lambda -i) \operatorname{Bi}\bigl((i \lambda +1)
   \sqrt[3]{i \lambda
   }\bigr)\right)\right)=0,
\end{multline}
where $\operatorname{Ai}(x)$ and $\operatorname{Bi}(x)$ are the Airy functions. In Table \ref{TabString1} we present the approximate eigenvalues produced by the modified SPPS method with $N=100$ and $M=10001$ and with the use of the spectral shift technique, the exact eigenvalues obtained from the characteristic equation \eqref{ExStringCharEq} with the help of Mathematica's function \texttt{FindRoot} and the absolute errors of the approximate eigenvalues compared to the exact ones. The eigenvalues are symmetric with respect to the imaginary axis, so we included only the eigenvalues with the positive real part. Note that our method allows one to obtain more eigenvalues, however Mathematica was unable to find more zeros of the characteristic equation.
\begin{table}
\centering
\begin{tabular}
{cccc}\hline
$n$ & $\lambda_n$ (our method) & $\lambda_{n}$ (exact) & Abs. error \\
\hline
1 & $0.724600759561354 + 0.465512975730082i$ & $0.724600759561355+ 0.465512975730082i$  & $1.1\cdot 10^{-15}$  \\
2 & $3.41348175703277 + 0.269073728680318i$ & $3.41348175703277+ 0.26907372868032i$  & $2.1\cdot 10^{-15}$ \\
3 &  $6.43085017426924 + 0.255763443512501i$ &  $6.43085017426926+ 0.255763443512497i$  & $2.4\cdot 10^{-14}$\\
4 &  $9.52497224975746 + 0.252665874553727i$ & $9.5249722497575+ 0.252665874553731i$  & $3.8\cdot 10^{-14}$\\
5 &  $12.6419970813013 + 0.251521276777511i$  & $12.6419970813014+ 0.251521276777512i$  & $4.8\cdot 10^{-14}$\\
7 &  $18.9002072286181 + 0.250683194824278i$ & $18.9002072286181+ 0.250683194824283i$  & $2.5\cdot 10^{-14}$\\
10 &  $28.3081715202515 + 0.250305060446283i$ & $28.3081715202511+ 0.250305060446279i$  & $3.4\cdot 10^{-13}$\\
15 &  $44.0040711901387 + 0.250126347925522i$ & $44.0040711901389 +     0.250126347925464i$  & $2.4\cdot 10^{-13}$\\
20 &  $59.7063095058408 + 0.250068647436092i$ & $59.7063095058413 +     0.250068647435942i$  & $5.8\cdot 10^{-13}$\\
\hline
\end{tabular}
\caption{The eigenvalues of the problem \eqref{ExStringEq} (Example \ref{ExString}).}
\label{TabString1}
\end{table}
\end{example}

\subsection{Spectral problems for Zakharov-Shabat systems}
Zakharov-Shabat systems arise in the application of the inverse scattering transform method to non-linear Schr\"{o}dinger equations, see, e.g., \cite{Ablowitz y Segur, Yang2010, Zakharov-Shabat}. In this subsection we follow definitions and results from the recent papers \cite{Kravchenko-Velasco, KTV}. We consider a generalized Zakharov-Shabat system
\begin{equation}\label{Eq ZS system}
    \begin{cases}
    v_1'=\lambda v_1+Pv_2,\\
    v_2'=-\lambda v_2-Qv_1,
    \end{cases}
\end{equation}
where $v_1$ and $v_2$ are unknown complex valued functions, $\lambda\in\mathbb{C}$ is a spectral parameter, $Q$ and $P$ are complex valued functions such that $Q$ does not vanish, $P$ is continuous and $Q$ is continuously differentiable on the domain of interest. Substituting $v_1=-\frac 1Q (v_2'+\lambda v_2)$ into the first equation in \eqref{Eq ZS system} we obtain an equation of the form
\begin{equation}
\left(\frac{1}{Q}v_{2}^{\prime}\right)^{\prime}+Pv_{2}=\lambda\frac{Q^{\prime}}{Q^{2}}v_{2}+\lambda^{2}\frac{1}{Q}v_{2}.\label{eq:ZS-pencil}
\end{equation}
Equation \eqref{eq:ZS-pencil} is of the form \eqref{pencil}, hence we can apply the results of Section \ref{Section3} to obtain the solution of the Zakharov-Shabat system.

Recall that the eigenvalue problem for the system \eqref{Eq ZS system} consists in finding such values of the spectral parameter $\lambda$ for which there exists a non-trivial Jost solution. In particular, when the potentials $Q$ and $P$ are compactly supported and non-vanishing on $[-a,a]$ (a situation which usually arises when truncating the infinitely supported and rapidly decreasing potentials) the eigenvalue problem reduces to finding such values of $\lambda$ (with $\operatorname{Re}\lambda >0$) for which there exists a solution of \eqref{Eq ZS system} on $(-a,a)$ satisfying the following boundary conditions (see, e.g., \cite{Kravchenko-Velasco})
\begin{align}
v_{1}\left(-a\right)&=1,\qquad  v_{2}\left(-a\right)=0,\label{eq: Jost -a}\\
v_{1}\left(a\right) &=0.\label{eq: Jost +a}
\end{align}

Let $f$ and $g$ be two particular solutions of \eqref{eq:ZS-pencil} for some $\lambda=\lambda_0$ satisfying the conditions of Theorem \ref{Th Modified SPPS copy(1)} and the solutions $u_1$ and $u_2$ be constructed by \eqref{u1u2PencilShifted} using $x_0=-a$ as the initial point in Definition \ref{Def Powers Pencil}. Then the general solution of \eqref{eq:ZS-pencil} has the form $v_2=c_1u_1+c_2u_2$ and it follows from \eqref{eq: Jost -a} and \eqref{initialPencil} that $c_1=0$, while from the boundary condition for the function $v_1=-\frac 1Q (v_2'+\lambda v_2)$ we obtain that $c_2=-1$. Hence due to \eqref{eq: Jost +a} the characteristic equation of the spectral problem reduces to
\[
0=v_1(a)=-\frac 1{Q(a)}\bigl(v_2'(a)+\lambda v_2(a)\bigr)=\frac 1{Q(a)}\bigl(u_2'(a)+\lambda u_2(a)\bigr).
\]
Multiplying both sides by $Q(a)$ we obtain that the eigenvalues of the spectral problem coincide with zeros of the characteristic function
\begin{equation}\label{eq:ZS char fun}
    \begin{split}
    \Phi(\lambda) &= \rho\sum_{n=0}^\infty (\lambda-\lambda_0)^n \bigl(g'(a) F_{2n}(a) - f'(a) G_{2n}(a)\bigr) + \lambda \sum_{n=0}^\infty (\lambda-\lambda_0)^n F_{2n+1}(a)\\
    &=\sum_{n=0}^\infty(\lambda-\lambda_0)^n\Bigl(\rho\bigl(g'(a)F_{2n}(a)-f'(a)G_{2n}(a)\bigr)+F_{2n-1}(a)+\lambda_0 F_{2n+1}(a)\Bigr).
    \end{split}
\end{equation}

\begin{example}\label{ExBronski}
Consider the following problem \cite{Bronski}
\begin{equation}\label{EqBronski}
    \begin{cases}
    i\varepsilon v'=q w + \lambda v,\\
    i\varepsilon w'=\bar q v - \lambda w,
    \end{cases}
\end{equation}
where the potential $q$ is given by
\[
q(x)=A(x)e^{i S(x)/\varepsilon},\qquad A(x)=S(x)=\operatorname{sech}(2x),
\]
$\bar q$ denotes the complex conjugate of $q$ and $\varepsilon$ is a small parameter. According to \cite{Bronski} the problem possesses a finite set of eigenvalues having a ``Y''-shape in the complex domain.

After division by $i\varepsilon$, \eqref{EqBronski} reduces to the Zakharov-Shabat system \eqref{Eq ZS system} with the spectral parameter $\widetilde\lambda = \lambda/i\varepsilon$. This problem was numerically solved in \cite[Example 4.10]{KTV} using machine-precision arithmetic by means of the original SPPS representation for several values of $\varepsilon\ge 0.063$. In \cite{Bronski} the graphs of the eigenvalues on the complex plane are presented for values of $\varepsilon$ as small as $0.023$. Such small values of $\varepsilon$ presented difficulties in \cite[Example 4.10]{KTV}. It was not possible to compute sufficiently many formal powers to obtain all the eigenvalues without using the spectral shift technique, the larger index formal powers became smaller than the smallest numbers in double precision. The spectral shift technique did not help either because of the rapid growth followed by the rapid decay of the particular solutions used for spectral shifts, similar difficulty as in the Coffey-Evans example \cite[Example 7.5]{KrT2013}. One possibility to overcome these difficulties in the framework of the original SPPS method consists in using arbitrary precision arithmetic. However even in this case the Clenshaw-Curtis integration formula allowed us to calculate only a few formal powers accurately, meanwhile the use of the Newton-Cottes integration formula led to elevated computational times.

The modified SPPS representation allowed us to overcome the main computation difficulty of the original SPPS representation --- nearly vanishing solutions. We truncated the potential to the segment $[-8,8]$ and computed two particular solutions of equation \eqref{eq:ZS-pencil} along with more than 2000 formal powers using the Clenshaw-Curtis integration formula. Such amount of formal powers is sufficient to obtain all eigenvalues of the problem \eqref{EqBronski} for all values of $\varepsilon$ reported in \cite{Bronski} directly from the truncated characteristic function \eqref{eq:ZS char fun}. We confirmed the smaller eigenvalues using the spectral shift method. For the larger eigenvalues the spectral shift method failed to produce reliable results with the parameters used because the particular solutions reveal a computationally difficult behavior, starting at 1 they grow to more than $10^{40}$ and than decay. All calculations were performed in Mathematica 8 using arbitrary precision arithmetic. On Figure \ref{ExBronskiEigenvalues} we present the graphs of the obtained eigenvalues for $\varepsilon=0.025$ and $\varepsilon=0.0223$, smallest values from \cite{Bronski}, and in Table \ref{TabBronski1} we present the approximate eigenvalues for $\varepsilon=0.025$.

\begin{figure}[h]
\centering
\includegraphics[
height=2.0in,
width=2.5in,
bb=0 0 180 144
]
{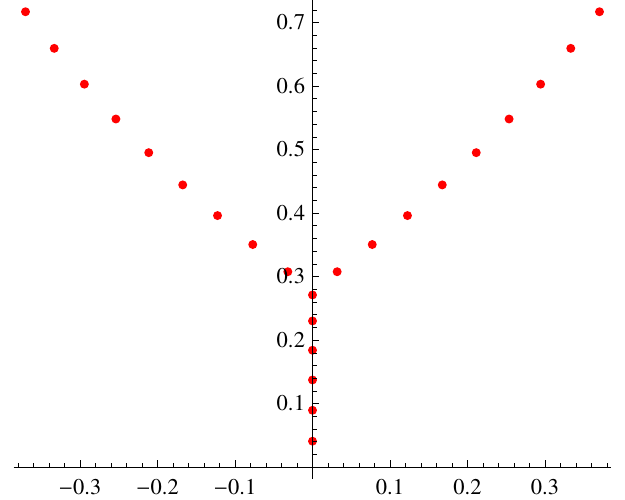}\quad
\includegraphics[
height=2.0in,
width=2.5in,
bb=0 0 180 144
]
{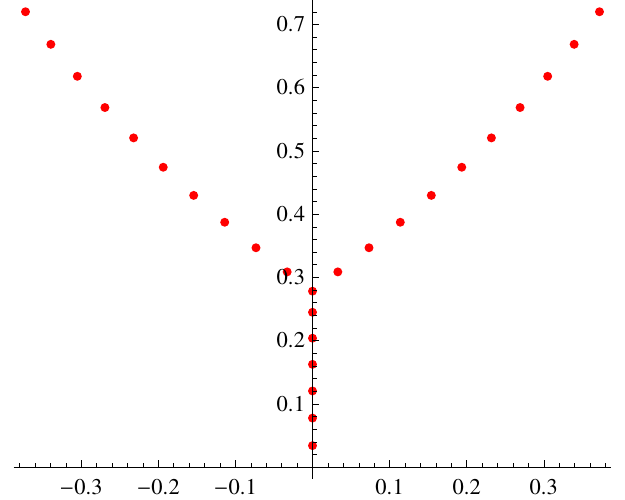}
\caption{Graphs of the eigenvalues of the problem \eqref{EqBronski} from  Example \ref{ExBronski} for $\varepsilon=0.025$ (on the left) and for $\varepsilon=0.0223$ (on the right).}\label{ExBronskiEigenvalues}
\end{figure}

\begin{table}
\centering
\begin{tabular}
{cc}\hline
$n$ & $\lambda_n$ (our method) \\
\hline
1 & $0.0407631404708347818i$\\
2 & $0.0894375732679124748i$\\
3 & $0.1371637819552007166i$\\
4 & $0.183974867838984848i$\\
5 & $0.230058316868567325i$\\
6 & $0.270992331790330981i$\\
7, 8 & $\pm 0.031821868157610443+0.307581835672991608i$\\
9, 10 & $\pm 0.077099873305919148 + 0.350414351766500910i$\\
11, 12 & $\pm 0.122549274277416703 + 0.396004265795121343i$\\
13, 14 & $\pm 0.167461129865269703 + 0.444273217295649703i$\\
15, 16 & $\pm 0.211303802258405430 + 0.495044232076272957i$\\
17, 18 & $\pm 0.253691811351361979 + 0.548028112642247410i$\\
19, 20 & $\pm 0.294386770221071665 + 0.602887000706132156i$\\
21, 22 & $\pm 0.333275636341542323 + 0.659282807236860406i$\\
23, 24 & $\pm 0.370339678528560140 + 0.716906655582204397i$\\
\hline
\end{tabular}
\caption{The eigenvalues of the problem \eqref{EqBronski} for $\varepsilon=0.025$ (Example \ref{ExBronski}).}
\label{TabBronski1}
\end{table}
\end{example}

\end{document}